\documentclass[12pt]{article}

\usepackage{a4wide}
\usepackage{fancyhdr}
\usepackage{graphicx}
\usepackage{amsfonts,amsmath,amssymb,amsthm}
\usepackage{ifthen}

\newcommand{\Dchaintwo}[4]{
\rule[-3\unitlength]{0pt}{8\unitlength}
\begin{picture}(14,5)(0,3)
\put(1,2){\ifthenelse{\equal{#1}{l}}{\circle*{2}}{\circle{2}}}
\put(2,2){\line(1,0){10}}
\put(13,2){\ifthenelse{\equal{#1}{r}}{\circle*{2}}{\circle{2}}}
\put(1,5){\makebox[0pt]{\scriptsize #2}}
\put(7,4){\makebox[0pt]{\scriptsize #3}}
\put(13,5){\makebox[0pt]{\scriptsize #4}}
\end{picture}}

\newcommand\ord{{\operatorname{ord}}}

\newcommand\Z{\mathbb{Z}}

\newcommand\Ly{\mathcal{L}}

\newcommand\SupW{[\Ly]^{(\N)}}

\newcommand\SupWL{[L]^{(\N)}}

\newcommand\N{\mathbb{N}}

\renewcommand{\k}{\Bbbk}
\newcommand\G{\Gamma}
\newcommand\Gh{\widehat{\Gamma}}

\renewcommand\o{\otimes}       
\newcommand\sw[1]{{}_{(#1)}}   

\newcommand\fv[1]{\vartheta_{{#1}}}      
\newcommand\fvsl[1]{x_{ {#1}}} 
\newcommand\fvsls[1]{X_{#1}} 
\newcommand\fvslbr[2]{[\fvsl{#1},\fvsl{#2}]}                  
\newcommand{\la}{\langle}      \newcommand{\ra}{\rangle}
  
\newcommand{\red}[1]{c_{ {#1}}}     
\newcommand{\redv}[1]{c_{{#1}}^{\rho}}     
\newcommand{\redbr}[2]{c_{{(#1|#2)}}}
\newcommand{\redbrv}[2]{c_{{(#1|#2)}}^{\rho}}
\newcommand{\redh}[1]{d_{{#1}}} 
\newcommand{\redhv}[1]{d_{{#1}}^{\rho}} 

\newcommand\kX{\k\la X\ra}

\newcommand{\kspan}{\operatorname{span}_\k}
\newcommand\Lv{\la \fvsls{L} \ra}
\newcommand\kLv{\k\la \fvsls{L} \ra}

\newcommand\X{\la X\ra}

\newcommand\Sh[3]{\mathrm{Sh}(#1)=(#2|#3)}
\newcommand\nSh[3]{\operatorname{Sh}(#1)\neq(#2|#3)}
\newcommand\idl[1]{I_{ \scriptscriptstyle{\! #1}}} 

\newcommand{\orddl}{\prec_{\diamond}}
\newcommand{\orddleq}{\preceq_{\diamond}}

\newcommand{\idlDL}[1]{I_{\orddl #1}}

\numberwithin{equation}{section}

\theoremstyle{plain}
\newtheorem{thm}{Theorem}[section]
\newtheorem{cor}[thm]{Corollary}
\newtheorem{lem}[thm]{Lemma}
\newtheorem{prop}[thm]{Proposition}

\theoremstyle{definition} 
\newtheorem{defn}[thm]{Definition}

\newtheorem{rem}[thm]{Remark}
\newtheorem{exs}[thm]{Examples}
\newtheorem{ex}[thm]{Example}
\newtheorem{notation}[thm]{Notation}

\title{A PBW basis criterion for 
pointed Hopf algebras\footnote{This work is part of the author's PhD thesis \cite{Helbig-PhD} written under the supervision of Prof.~H.-J.~Schneider.}}
\author{Michael Helbig\footnote{eMail: \texttt{michael@helbig123.de}}}
\date{\today}

\begin{document}

\maketitle


\begin{abstract}
\noindent
We give a necessary and sufficient PBW basis criterion for Hopf algebras generated by skew-primitive elements and abelian group of group-like elements with action given via characters. This class of pointed Hopf algebras has shown great importance in the classification theory and can be seen as generalized quantum groups.

\noindent
We apply the criterion to classical examples and liftings of Nichols algebras which were determined in \cite{Helbig-Lift}.
\newline


\end{abstract}



\section*{Introduction}

In the famous Poincar\'{e}-Birkhoff-Witt theorem for universal enveloping algebras of finite-dimensional Lie algebras a class of new bases appeared. Since then many PBW theorems for more general situations were discovered. We want to name those for quantum groups: Lusztig's axiomatic approach \cite{LusztPBW,Luszt3} and 
Ringel's approach via Hall algebras \cite{RingelPBW}. Let us also mention the work of Berger \cite{BergerPBW}, Rosso \cite{RossoPBW}, and Yamane \cite{YamanePBW}.

Our starting point of view is the following: Part of the classification program of finite-dimensional pointed Hopf algebras with the lifting method of Andruskiewitsch and Schneider \cite{AS-p3} is the knowledge of the dimension resp.~a basis of the deformations of a Nichols algebra (the so-called \emph{liftings}). Another aspect is to find the redundant relations in the ideal. These liftings are among the class we consider here. We want to present a necessary and sufficient PBW basis criterion for Hopf algebras generated by skew-primitive elements and abelian group of group-like elements with action given via characters. This class contains all quantum groups, Nichols algebras and their liftings and it is conjectured that any finite-dimensional pointed Hopf algebra over the complex numbers is of that form.

The very general and for us important work is \cite{KhPBW}, where a PBW theorem for the here considered class of Hopf algebras  is formulated: Kharchenko shows in \cite[Thm.~2]{KhPBW} these Hopf algebras 
have a PBW basis in special $q$-commutators, namely the \emph{hard super letters} coming from the theory of Lyndon words, see Section \ref{SectCharHA}. However, the definition of \emph{hard} is not constructive (see also \cite{BokutUnsolv,BokutBook} for the word problem for Lie algebras) and in view of treating concrete examples there is a lack of deciding whether a given set of iterated $q$-commutators establishes a PBW basis.

On the other hand the diamond lemma \cite{BergmanDL} (see also Section \ref{SectDL}, Theorem \ref{ThmDL})  is a very general method to check whether an associative algebra given in terms of generators and relations has a certain basis, or equivalently the relations form a Gr\"{o}bner basis. As mentioned before, we construct such a Gr\"{o}bner basis for a character Hopf algebra in Theorem \ref{PropIdealCharHopfAlg} and give a necessary and sufficient criterion for a set of super letters being a PBW basis, see Theorem \ref{ThmPBWCrit}. The PBW Criterion \ref{ThmPBWCrit} is formulated in the languague of $q$-commutators. This seems to be the natural setting, since the criterion involves only $q$-commutator identities of Proposition \ref{PropqCommut}; as a side effect we find redundant relations.

The main idea is to combine the diamond lemma with the combinatorial theory of Lyndon words resp.~super letters and the $q$-commutator calculus of Section \ref{SectqCommutCalc}. In order to apply the diamond lemma we give a general construction to identify a smash product with a quotient of a free algebra, see Proposition \ref{IdentFreeAlg} in Section \ref{SectIdentFreeAlg}.

Further the PBW Criterion \ref{ThmPBWCrit} is a generalization of \cite{BergerPBW} and \cite[Sect.~4]{AS-Class} in the following sense: In  \cite{BergerPBW}  a condition involving the $q$-Jacoby identity for the generators $x_i$ occurs (it is called ``$q$-Jacobi sum''). However, this condition can be formulated more generally for iterated $q$-commutators (not only for $x_i$), so also higher than quadratic relations can be considered. The intention of \cite{BergerPBW} was a $q$-generalization of the classical PBW theorem, so powers of $q$-commutators are not covered at all and also his algebras do not contain a group algebra. On the other hand, \cite[Sect.~4]{AS-Class} deals with powers of $q$-commutators (root vector relations) and also involves 
the group algebra. But here it is assumed that 
the powers of the commutators lie in the group algebra and fulfill a certain centrality condition. As mentioned above these assumptions are in general not preserved; in the PBW Criterion \ref{ThmPBWCrit} the centrality condition is replaced by a more general condition involving the restricted $q$-Leibniz formula of Proposition \ref{PropqCommut}.

This work is organized as follows: In Section 1 we develop a general calculus for $q$-commutators in an arbitrary algebra, which is needed throughout the thesis; new formulas for $q$-commutators are found in Proposition \ref{PropqCommut}.
We recall in Section \ref{SectLyndW} the theory of Lyndon words, super letters and super words. We show that the set of all super words can be seen indeed as a set of words, i.e., as a free monoid. In Section \ref{SectCharHA} we recall the result of \cite{Helbig-Presentation} about a structural description of the here considered Hopf algebras, in terms of generators and relations. 
With this result we are able to formulate in Section 4 the main result of this work, namely the PBW basis criterion. Sections 5 to 7 are dedicated to the proof of the criterion.
Finally in Sections 8 and 9 we apply the PBW Criterion \ref{ThmPBWCrit} to classical examples and the liftings of Nichols algebras obtained in \cite{Helbig-Lift}.

    \section{$q$-commutator calculus}\label{SectqCommutCalc}
In this section let $A$ denote an arbitrary algebra over a field $\k$ of characteristic $\operatorname{char}\k=p\ge 0$.  The main result of this section is Proposition \ref{PropqCommut}, which states important $q$-commutator formulas in an arbitrary algebra.

\subsection{$q$-calculus} For every $q\in \k$ we define for $n\in\N$ and $0\le i\le n$ the \emph{$q$-numbers} $(n)_q := 1+q+q^2+\ldots+q^{n-1}$, the \emph{$q$-factorials} 
$(n)_q ! := (1)_q (2)_q \ldots (n)_q,$ 
and the \emph{$q$-binomial coefficients} $\tbinom{n}{i}_q:=\frac{(n)_q !}{(n-i)_q !(i)_q !}.$
Note that the latter right-handside is well-defined since it is a polynomial over $\Z$ evaluated in $q$.
We denote the \emph{multiplicative order} of any $q\in  \k^{\times}$ by $\ord q$.
If $q\in \k^{\times}$ and $n>1$, then 
\begin{align}\label{qBinomCoeffZero}
\binom{n}{i}_q =0\text{ for all }1\le i\le n-1 \Longleftrightarrow 
\begin{cases}
 \ord q = n ,&\mbox{if }\operatorname{char}\k=0\\
p^k\ord q =n\text{ with }k\ge 0 ,&\mbox{if }\operatorname{char}\k=p>0,
\end{cases}
\end{align}
see \cite[Cor. 2]{Radford}. Moreover for $1\le i\le n$ there are the \emph{$q$-Pascal identities}
\begin{align}\label{PascalDreieck}
q^i \binom{n}{i}_q + \binom{n}{i-1}_q = 
\binom{n}{i}_q +q^{n+1-i}\binom{n}{i-1}_q =\binom{n+1}{i}_q,
\end{align}
and the \emph{$q$-binomial theorem}: For $x,y\in A$ and $q\in \k^{\times}$ with $yx=qxy$ we have
\begin{align}\label{qBinomThm}
(x+y)^n=\sum_{i=0}^{n} \tbinom{n}{i}_q x^i y^{n-i}.
\end{align}
Note that for $q=1$ these are the usual notions.

\subsection{$q$-commutators} For all $a,b\in A$ and $q\in \k$ we define the \emph{$q$-commutator}
$$
[a,b]_{q} := ab-qba.
$$
The $q$-commutator is bilinear. If $q=1$ we get the classical commutator of an 
algebra. If $A$ is graded and $a,b$ are homogeneous 
elements, then there is a natural choice for the $q$. 
We are interested in the following special case:
\begin{ex}\label{qCommutExkX}\label{qCommutEx} 
Let $\theta\ge 1$, $X=\{ x_1,\ldots,x_{\theta}\}$, $\X$ the free monoid and $A=\kX$ the free $\k$-algebra. For an abelian group $\G$ let $\widehat{\G}$ be the character group, $g_1,\ldots,g_{\theta}\in\G$ and $\chi_1,\ldots,\chi_{\theta}\in \widehat{\G}$. If we define the two monoid maps 
$$
\deg_{\G}:\X\rightarrow \G,\ \deg_{\G}(x_i):=g_i\quad\text{and}\quad \deg_{\Gh}:\X\rightarrow \Gh,\ \deg_{\Gh}(x_i):=\chi_i,
$$ 
for all $1\le i\le\theta$, then $\kX$ is $\G$- and $\Gh$-graded. 

Let $a\in \kX$ be $\G$-homogeneous and $b\in \kX$ be $\widehat{\G}$-homogeneous. We set 
$$
g_a:=\deg_{\G}(a), \quad \chi_b:=\deg_{\Gh}(b),\quad\text{and}\quad q_{a,b}:=\chi_b(g_a).
$$ 
Further we define $\k$-linearly on $\kX$ the $q$-commutator
\begin{align}\label{qCommutExkXDefn}
 [a,b] := [a,b]_{q_{a,b}}.
\end{align}
Note that $q_{a,b}$ is a bicharacter on the homogeneous elements 
and depends only on the values 
$$
q_{ij}:=\chi_j(g_i)\text{ with }1\le i,j\le\theta.$$
For example $[x_1,x_2]=x_1x_2-\chi_2(g_1)x_2x_1=x_1x_2-q_{12}x_2x_1$. Further if $a,b$ are $\Z^{\theta}$-homogeneous they are both $\G$- and $\Gh$-homogeneous. In this case we can build iterated $q$-commutators, like $\bigl[x_1,[x_1,x_2]\bigr] =x_1[x_1,x_2]-\chi_1\chi_2(g_1)[x_1,x_2]x_1= x_1[x_1,x_2]-q_{11}q_{12}[x_1,x_2]x_1$.

\end{ex}

Later we will deal with algebras which still are  $\widehat{\G}$-graded,  but not $\G$-graded such that Eq.~\eqref{qCommutExkXDefn} is not well-defined. However, the $q$-commutator calculus, which we next want to develop, will be a major tool for our calculations such that we need the general definition with the $q$ as an index.

\begin{prop}\label{PropqCommut} For all $a,b,c,a_i,b_i\in A$, $q,q',q'',q_i,\zeta\in \k$, $1\le i\le n$ and $r\ge 1$ we have:

\emph{(1) $q$-derivation properties:}
\begin{align*}
   &[a,bc]_{qq'}=[a,b]_{q} c + q b [a,c]_{q'}, \qquad
   [ab,c]_{qq'}=a [b,c]_{q'} + q' [a,c]_{q} b,\\
   &[a,b_1\ldots b_n]_{q_1\ldots q_n}=\sum_{i=1}^{n}q_1\ldots q_{i-1} b_1\ldots b_{i-1}[a,b_i]_{q_i}b_{i+1}\ldots b_n,\\ 
   &[a_1\ldots a_n,b]_{q_1\ldots q_n}=\sum_{i=1}^{n}q_{i+1}\ldots q_{n} a_1\ldots a_{i-1}[a_i,b]_{q_i}a_{i+1}\ldots a_n.
\end{align*}

\emph{(2) $q$-Jacobi identity:}
\begin{align*}
   \bigl[[a,b]_{q'}, c\bigr]_{q''q}&=\bigl[a,[b,c]_{q}\bigr]_{q'q''} -q' b [a,c]_{q''}+ q[a,c]_{q''} b.
\end{align*}

\emph{(3) $q$-Leibniz formulas:}
\begin{align*}
  [a,b^r]_{q^r} &= \sum_{i=0}^{r-1} q^i \tbinom{r}{i}_{\zeta} 
                           b^i \bigl[\ldots\bigl[[a,\underbrace{b]_q , b\bigr]_{q\zeta}\ldots,b}_{r-i}\bigr]_{q\zeta^{r-i-1}},\\
  [a^r,b]_{q^r} &= \sum_{i=0}^{r-1} q^i \tbinom{r}{i}_{\zeta} 
                      \bigl[\underbrace{a,\ldots\bigl[a,[a}_{r-i},b]_q\bigr]_{q\zeta}\ldots\bigr]_{q\zeta^{r-i-1}}a^i.
\end{align*}

\emph{(4) restricted $q$-Leibniz formulas:} If $\operatorname{char} k=0$ and $\ord\zeta = r$, or $\operatorname{char} \k =p >0$ and $p^k \ord\zeta =r$ , then
\begin{align*}
 [a,b^r]_{q^r} &= \bigl[\ldots\bigl[[a,\underbrace{b]_q,b\bigr]_{q\zeta}\ldots,b}_{r}\bigr]_{q\zeta^{r-1}},\\
 [a^r,b]_{q^r} &= \bigl[\underbrace{a,\ldots\bigl[a,[a}_{r},b]_q\bigr]_{q\zeta}\ldots\bigr]_{q\zeta^{r-1}}.
\end{align*}
\end{prop} 
\begin{proof} (1) The first part is a direct calculation, e.g.
\begin{align*}
 [a,bc]_{qq'}=abc- qq'bca =abc-qbac+ qbac- qq'bca =[a,b]_{q} c + q b [a,c]_{q'}.
\end{align*}
The second part follows by induction.

\noindent
(2) Using the $\k$-linearity and (1) we get the result immediately.

\noindent
(3) By induction on $r$: $r=1$ is obvious, so let $r\ge 1$. Using (1) we get 
$$
[a,b^{r+1}]_{q^{r+1}}=[a,b^r b]_{q^{r}q} = [a,b^r]_{q^r} b + q^r b^r[a,b]_q.$$
 By induction assumption 
$
[a,b^{r}]_{q^r} b= \sum_{i=0}^{r-1} q^i \tbinom{r}{i}_{\zeta} b^i \bigl[\ldots\bigl[[a,\underbrace{b]_q , b\bigr]_{q\zeta}\ldots,b}_{r-i}\bigr]_{q\zeta^{r-i-1}} b,$ where 
\begin{multline*}
b^i\bigl[\ldots\bigl[[a,\underbrace{b]_q , b\bigr]_{q\zeta}\ldots,b}_{r-i}\bigr]_{q\zeta^{r-i-1}} b =\\ b^i\bigl[\ldots\bigl[[a,\underbrace{b]_q , b\bigr]_{q\zeta}\ldots,b}_{r+1-i}\bigr]_{q\zeta^{r-i}}+q\zeta^{r-i} b^{i+1}\bigl[\ldots\bigl[[a,\underbrace{b]_q , b\bigr]_{q\zeta}\ldots,b}_{r-i}\bigr]_{q\zeta^{r-i-1}}.
\end{multline*}
In total we get
\begin{multline*}
[a,b^{r+1}]_{q^{r+1}} =
\sum_{i=0}^{r} q^i \tbinom{r}{i}_{\zeta} b^i\bigl[\ldots\bigl[[a,\underbrace{b]_q , b\bigr]_{q\zeta}\ldots,b}_{r+1-i}\bigr]_{q\zeta^{r-i}}\\
+ \sum_{i=0}^{r-1} q^{i+1} \tbinom{r}{i}_{\zeta}\zeta^{r-i} b^{i+1}\bigl[\ldots\bigl[[a,\underbrace{b]_q , b\bigr]_{q\zeta}\ldots,b}_{r-i}\bigr]_{q\zeta^{r-i-1}}.
\end{multline*}
Shifting the index of the second sum and using Eq.~\eqref{PascalDreieck} for $\zeta$ we get the formula. The second formula is proven in the same way. (4) Follows from (3) and Eq.~\eqref{qBinomCoeffZero}.
\end{proof}



\section{Lyndon words and $q$-commutators}\label{SectLyndW}

In this section we recall the theory of Lyndon words \cite{Lothaire,Reut} as far as we are concerned and then introduce the notion of super letters and super words \cite{KhPBW}.

\subsection{Words and the lexicographical order}
Let $\theta\ge 1$, $X=\{x_1,x_2,\ldots,x_{\theta}\}$ be a finite totally ordered set by $x_1<x_2<\ldots <x_{\theta}$, and $\X$ the free monoid; we think of $X$ as an alphabet and of $\X$ as the words in that alphabet including the empty word $1$. For a word $u=x_{i_1}\ldots x_{i_n}\in\X$ we define $\ell(u):=n$ and call it the \emph{length} of $u$. 

The \emph{lexicographical order} $\le$ on $\X$ is defined for $u,v\in\X$ by $u<v$ if and only if 
either $v$ begins with $u$, i.e., $v=uv'$ for some $v'\in\X\backslash\{1\}$, or if there are $w,u',v'\in \X$, $x_i,x_j\in X$ 
such that $u=wx_iu'$, $v=wx_jv'$ and $i<j$. E.g., $x_1<x_1x_2<x_2$. 

\subsection{Lyndon words and the Shirshov decomposition}
A word $u\in\X$ is called a \emph{Lyndon word} if $u\neq 1$ and $u$ is smaller than any of its proper endings, i.e., for all $v,w\in\X\backslash\{1\}$ such that $u=vw$ we have $u<w$. We denote by 
$$
\Ly:=\{u\in\X\,|\, u \text{ is a Lyndon word}\}
$$ the set of all Lyndon words. For example $X\subset\Ly$, but $x_i^n\notin \Ly$ for all $1\le i\le \theta$ and $n\ge 2$. 
Also $x_1x_2$, $x_1x_1x_2$, $x_1x_2x_2$,  $x_1x_1x_2x_1x_2\in\Ly$.

For any $u\in\X\backslash X$ we call the decomposition $u=vw$ with $v,w\in \X\backslash\{1\}$ such that $w$ is the minimal (with respect to the lexicographical order) ending  the \emph{Shirshov decomposition} of the word $u$. We will write in this case $$\Sh{u}{v}{w}.$$ E.g., $\Sh{x_1x_2}{x_1}{x_2}$, $\Sh{x_1x_1x_2x_1x_2}{x_1x_1x_2}{x_1x_2}$, $\nSh{x_1x_1x_2}{x_1x_1}{x_2}$.
If $u\in\Ly\backslash X$, this is equivalent to $w$ is the longest proper ending of $u$ such that $w\in\Ly$.

\begin{defn}
We call a subset $L\subset \Ly$  \emph{Shirshov closed} 
if
  $X\subset L$,
  and for all $u\in L$ with $\Sh{u}{v}{w}$ also $v,w\in L$.
\end{defn}

For example $\Ly$ is Shirshov closed, and if $X=\{x_1,x_2\}$, then $\{x_1,x_1x_1x_2,x_2\}$ is  not Shirshov closed,  whereas $\{x_1,x_1x_2,x_1x_1x_2,x_2\}$ is.

\subsection{Super letters and super words} Let the free algebra $\kX$ be graded as in Section \ref{qCommutExkX}. For any $u\in\Ly$ we define recursively on $\ell(u)$ the map 
\begin{align}\label{DefnSuperLett}
[\,.\,]:\Ly\rightarrow\kX,\quad u \mapsto [u].
\end{align}
If $\ell(u)=1$, then set $[x_i]:=x_i$ for all $1\le i\le\theta$. Else if $\ell(u)>1$ and $\Sh{u}{v}{w}$ we define $[u]:=\bigl[[v],[w]\bigr]$. This map is well-defined since inductively all $[u]$ are $\Z^{\theta}$-homogeneous such that we can build iterated $q$-commutators; see Section \ref{qCommutExkX}. The elements $[u]\in\kX$ with $u\in\Ly$ are called \emph{super letters}. E.g. $[x_1x_1x_2x_1x_2]=\bigl[[x_1x_1x_2],[x_1x_2]\bigr]=\bigl[[x_1,[x_1,x_2]],[x_1,x_2]\bigr]$.
If $L\subset \Ly$ is Shirshov closed then the subset of $\kX$
$$
[L]:=\bigl\{[u]\,\big|\, u\in L\bigr\}
$$ 
is a set of iterated $q$-commutators. Further 
$
[\Ly]=\bigl\{[u]\,\big|\, u\in\Ly\bigr\}
$
is the set of all super letters and the map 
$[\,.\,]:\Ly\rightarrow[\Ly]$ is a bijection, which follows from \cite[Lem.~2.5]{Helbig-Presentation}. Hence we can define an order $\le$ of the super letters $[\Ly]$ by 
$$
[u]< [v]:\Leftrightarrow u<v,
$$ thus $[\Ly]$ is a new alphabet containing the original alphabet $X$; so the name ``letter'' makes sense. Consequently, products of super letters are called \emph{super words}. We denote 
$$
[\Ly]^{(\N)}:=\bigl\{[u_1]\ldots [u_n]\,\bigl|\,n\in\N,\, u_i\in\Ly\bigr\}
$$
the subset of $\kX$ of all super words. Any super word has a unique factorization in super letters \cite[Prop.~2.6]{Helbig-Presentation}, hence we can define the lexicographical order on $\SupW$, as defined above on regular words. We denote it also by $\le$.

\subsection{A well-founded ordering of super words}\label{SectWellFoundOrder} The \emph{length} of a super word $U=[u_1][u_2]\ldots[u_n]\in\SupWL$ is defined as 
$
\ell(U):=\ell(u_1u_2\ldots u_n).
$
\begin{defn}
For $U, V\in \SupW$ we define $U\prec V$ by 
\begin{itemize}
\item $\ell(U)<\ell(V)$, or
\item $\ell(U)=\ell(V)$  and $U>V$ lexicographically in $\SupW$.
\end{itemize}
\end{defn}
This defines a total ordering of $\SupW$ with minimal element $1$. As $X$ is assumed to be finite, there are only finitely many super letters of a given length. Hence every nonempty subset of $\SupW$ has a minimal element, or equivalently, $\preceq$ fulfills the descending chain condition: $\preceq$ is \emph{well-founded}. This makes way for inductive proofs on $\preceq$.

\subsection{The free monoid $\Lv$}\label{SectFreeMonoidXL}

Let  $L\subset\Ly$. We want to stress the two different aspects of a super letter $[u]\in [L]$: 
\begin{itemize}
 \item  On the one hand it is by definition a polynomial $[u]\in\kX$.
\item On the other hand, as we have seen, it is a letter in the alphabet $[L]$.
\end{itemize}
To distinguish between these two point of views we define for the latter aspect a new alphabet corresponding to the set of super letters $[L]$:
To be technically correct we regard the free monoid $\la 1,\ldots,\theta\ra$ of the ciphers $\{1,\ldots,\theta\}$ (telephone numbers), together with the trivial bijective monoid map
 $
\nu: \la x_1,\ldots,x_{\theta}\ra \rightarrow  \la 1,\ldots,\theta\ra, \ x_{i}\mapsto i\ \text{ for all }1\le i\le \theta.
$
Hence we can transfer the lexicographical order to $\la 1,\ldots,\theta\ra$.  The image $\nu(\Ly)\subset\la 1,\ldots,\theta\ra$  
can be seen as the set of ``Lyndon telephone numbers''. We define the set 
$$
X_{L}:=
\{\fvsl u\ |\ u\in\mathcal \nu(L)\}.
$$
Note that if  $X\subset L$ (e.g.~$L\subset\Ly$ is Shirshov closed), then  $X\subset X_L$. 
E.g., if $X=\{x_1,x_2\}\subset L=\{x_1,x_1x_2,x_2\}$ then $\nu(L)=\{1,12,2\}$ and $X\subset X_L=\{x_1,x_{12},x_2\}$.

\begin{notation}
From now on we will not distinguish between $L$ and $\nu(L)$ and write for example $x_u$ instead of $x_{\nu(u)}$ for $u\in L$. In this manner we will also write $g_{\nu(u)},\chi_{\nu(u)}$ equivalently for $g_u,\chi_{u}$ if $u\in L$, as defined in Example \ref{qCommutExkX}. 
E.g.~$g_{112}=g_{x_1x_1x_2}=g_{x_1}g_{x_1}g_{x_2}=g_{1}g_{1}g_{2}$, $\chi_{112}=\chi_{x_1x_1x_2}=\chi_{x_1}\chi_{x_1}\chi_{x_2}=\chi_{1}\chi_{1}\chi_{2}$.
\end{notation}
\noindent
As seen in \cite[Prop.~2.6]{Helbig-Presentation} we have the bijection of super words and the free monoid $\Lv$
\begin{align}\label{DefnUv}
\rho:\SupWL\rightarrow\Lv,\quad 
\rho\bigl([u_1]
\ldots[u_n]\bigr):=\fvsl{u_1}
\ldots\fvsl{u_n}. 
\end{align}
E.g., $[x_1x_2x_2][x_1x_2]\stackrel{\rho}{\mapsto}\fvsl{122}\fvsl{12}$. Hence we can transfer all orderings to $\Lv$: For all $U,V\in\Lv$ we set 
$$
 \ell(U):=\ell(\rho^{-1}(U)), \quad
U<V  :\Leftrightarrow \rho^{-1}(U)< \rho^{-1}(V), \quad
U\prec V  :\Leftrightarrow  \rho^{-1}(U)\prec \rho^{-1}(V).
$$

  \section{A class of pointed Hopf algebras}\label{SectCharHA}
In this chapter we deal with the class of pointed Hopf algebras for which we give the PBW basis criterion. Let us recall the notions and results of \cite[Sect.~3]{KhPBW}: A Hopf algebra $A$ is called a \emph{character Hopf algebra} if it is generated as an algebra by 
elements $a_1,\ldots ,a_{\theta}$ and an abelian group $G(A)=\G$ of all group-like elements such that  for all $1\le i\le \theta$ there are $g_i\in\G$ and $\chi_i\in\Gh$ with
\begin{align*}
\Delta(a_i)=a_i\o 1 +g_i\o a_i\qquad \text{and}\qquad
 ga_i = \chi_i(g)a_i g.
\end{align*}
As mentioned in the introduction this covers a wide class of examples of Hopf algebras. 

\begin{thm}\emph{\cite[Thm.~3.4]{Helbig-Presentation}} 
\label{PropIdealCharHopfAlg}
If $A$ is a character Hopf algebra, then
$$
A\cong (\kX \# \k[\G])/I,
$$
where the smash product $\kX\#\k[\G]$ and the ideal $I$ are constructed in the following way:
\end{thm}

\subsection{The smash product $\kX\# \k[\G]$}\label{SectSmashProto}
Let $\kX$ be $\G$- and $\Gh$-graded as in Section \ref{qCommutExkX}, and $\k[\G]$ be endowed with the usual bialgebra structure $\Delta(g)=g\o g$ and $\varepsilon(g)=1$ for all $g\in\G$. 
Then we define 
$$
g\cdot x_i := \chi_i(g)x_i,\ \text{ for all }1\le i\le\theta.
$$
In this case, $\kX$ is a $\k[\G]$-module algebra and we calculate $gx_i=\chi_i(g)x_ig$, $gh=hg=\varepsilon(g)hg$ in $\kX\# \k[\G]$.
Further $\kX\# \k[\G]$ is a Hopf algebra with structure determined for all $1\le i \le\theta$ and $g\in\G$ by 
\begin{align*}
 \Delta(x_i):=x_i\o 1 + g_i\o x_i\qquad\text{and}\qquad  \Delta(g):=g\o g.
\end{align*}

\subsection{Ideals associated to Shirshov closed sets}\label{SectIdealOfCharHA}
In this subsection we fix a Shirshov closed $L\subset\Ly$. 
We want to introduce the following notation 
for an $a\in \kX\# \k[\G]$ and  
$W\in\SupW$: 
We will write $a\prec_L W$ (resp. $a\preceq_L W$), if 
$a$ is a linear combination of 
\begin{enumerate}
 \item[\textbullet] $U\in\SupWL$ with $\ell(U)=\ell(W)$, $U>W$ (resp. $U\ge W$), and
\item[\textbullet]  $Vg$ with $V\in\SupWL$, $g\in\G$,  $\ell(V)<\ell(W)$.
\end{enumerate}

Furthermore,  we set for each $u\in L$
either $N_u:=\infty$ or $N_u:=\ord q_{u,u}$  (resp.~$N_u:=p^k\ord q_{u,u}$ with $k\ge 0$ if $\operatorname{char} \k=p>0$) and we want to distinguish the following two sets of words depending on $L$:
\begin{align*}
C(L) &:= \bigl\{w\in\X\backslash L \ |\ \exists u,v\in L:  w=uv,\ u<v,\  \text{and}\ \Sh{w}{u}{v}\bigr\},\\
D(L) &:= \bigl\{u\in L \ |\ N_u<\infty\}.
\end{align*}
Note that $C(L)\subset \Ly$ and $D(L)\subset L\subset \Ly$ are sets of Lyndon words. For example, if $L=\{x_1,x_1x_1x_2,x_1x_2,x_2\}$, then  $C(L)=\{x_1x_1x_1x_2$, $x_1x_1x_2x_1x_2$, $x_1x_2x_2\}$.

Moreover, let $\red{w} \in (\kX\# \k[\G])^{\chi_{w}}$ for all $w\in C(L)$ such that $\red{w}\prec_L [w]$;
and let $\redh{u} \in (\kX\# \k[\G])^{\chi_u^{N_u}}$ for all $u\in D(L)$ such that $\redh{u}\prec_L [u]^{N_u}$. 
Then let $I$ be the $\Gh$-homogeneous ideal of $\kX\# \k[\G]$ generated by the following elements: 
\begin{align}
[w]         - \red{w}&    &&\text{for all } w\in C(L) ,\label{ThmPBWCritIdluv}\\
 [u]^{N_u}   - \redh{u}&    &&\text{for all }  u\in D(L).\label{ThmPBWCritIdluN}
\end{align}

\section{A PBW basis criterion}\label{SectPBWCrit}
In this section we want to state a PBW basis criterion which is applicable for any character Hopf algebra.
Suppose we have a smash product $\kX\# \k[\G]$ together with an ideal $I$ as in Sections \ref{SectSmashProto} and \ref{SectIdealOfCharHA}.

At first we need to define several algebraic objects for the formulation of the PBW Criterion \ref{ThmPBWCrit}. The main idea is not to work in the free algebra $\kX$ but in the free algebra $\kLv$ where  $\Lv$ is the free monoid of Section \ref{SectFreeMonoidXL}.

\subsection{The free algebra $\kLv$ and $\kLv\# \k[\G]$} 
In Section \ref{SectFreeMonoidXL} we associated to a super letter $[u]\in[L]$ a new variable $x_u\in X_L$, where $X_L$ contains $X$. Hence the free algebra $\kLv$ also contains $\kX$.
We define the action of $\G$ on $\kLv$ and $q$-commutators by
\begin{align*}
g\cdot \fvsl{u}&:=\chi_u(g) \fvsl{u} && \text{for all }g\in\G,u\in L,\\
[\fvsl u,\fvsl v]&:= \fvsl u\fvsl v -q_{u,v} \fvsl v\fvsl u && \text{for all }u,v\in L.
\end{align*}
In this way $\kLv$ becomes a $\k[\G]$-module algebra and $g\fvsl{u}=\chi_u(g) \fvsl{u}g$ in the smash product $\kLv\# \k[\G]$.

\subsection{The subspace $I_{\prec U}\subset \kLv\# \k[\G]$}\label{SectSubspaceofIdeal}

Via $\rho$ of  Eq.~\eqref{DefnUv} we now define   certain elements of $\kLv\# \k[\G]$: For all $w\in C(L)$ resp.~$u\in D(L)$  we write $\red{w}=\sum \alpha U+\sum \beta Vg\prec_L [w]
$ resp.~$\redh{u}=\sum \alpha' U'+\sum \beta' V'g'\prec_L [u]^{N_u}
$, with $\alpha,\alpha',\beta,\beta'\in\k$ and $U,U',V,V'\in \SupWL$ 
(such decompositions may not be  unique; we just fix one). Then we define in $\k\Lv\#\k[\G]$
\begin{align*}
\redv{w} :=\sum \alpha\rho(U)+\sum \beta \rho(V) g 
\quad\text{resp.}\quad
\redhv{u}:=\sum \alpha' \rho(U')+\sum \beta' \rho(V')g'.
\end{align*}
For all $u,v\in L$ with $u<v$ we define elements $\redbrv{u}{v}\in \kLv\# \k[\G]$: If $w=uv$ and $\Sh{w}{u}{v}$ we set
$$
\redbrv{u}{v}:=\begin{cases}   \fvsl{w}  ,  &\mbox{if }w\in L, \\ 
                              \redv{w}   ,  &\mbox{if }w\notin L.
              \end{cases}
$$
Else if $\nSh{w}{u}{v}$ let $\Sh{u}{u_1}{u_2}$. Then we define inductively on the length of $\ell(u)$ 
\begin{align}\label{RedCommutDefnVariableNotSh}
\redbrv{u}{v} :=
     \partial_{u_1}^{\rho}(\redbrv{u_2}{v}) +q_{u_2,v}\redbrv{u_1}{v}\fvsl{u_2} -q_{u_1,u_2}\fvsl{u_2}\redbrv{u_1}{v},
\end{align}
where 
$\partial_{u_1}^{\rho}$ is defined $\k$-linearly by
\begin{align*}
&\partial_{u_1}^{\rho}(\fvsl{l_{1}}\ldots \fvsl{l_n}):=\redbrv{u_1}{l_1}\fvsl{l_{2}}\ldots \fvsl{l_n} + \sum_{i=2}^{n} q_{u_1,l_1\ldots l_{i-1}} \fvsl{l_1}\ldots \fvsl{l_{i-1}} \bigl[\fvsl{u_1},\fvsl{l_i}\bigr] \fvsl{l_{i+1}}\ldots \fvsl{l_n},\\
&\partial_{u_1}^{\rho}(\rho(V)g):= \bigl[\fvsl{u_1},\rho(V)\bigr]_{q_{u_1,u_2v}\chi_{u_1}(g)} g.
\end{align*}

For any $U\in \Lv$ let $I_{\prec U}$  denote the subspace of $\kLv\# \k[\G]$ spanned by the elements
\begin{align*}
&Vg\bigl([\fvsl u,\fvsl v] -\redbrv{u}{v} \bigl)Wh & &\text{for all } u,v\in L,u<v, \\
&V'g'\bigl(\fvsl{u}^{N_u} - \redhv{u}  \bigr)W' h'  &  &\text{for all }u\in L, N_u<\infty
\end{align*}
with $V,V',W,W'\in\Lv$, $g,g',h,h'\in\G$ such that  
$$V\fvsl u\fvsl v W\prec U\quad\text{ and }\quad V'\fvsl{u}^{N_u} W'\prec U.$$

Finally we want to define the following elements of $\kLv\# \k[\G]$ for $u,v,w\in L$, $u<v<w$, resp.~$u\in L$, $N_u<\infty$, $u\le v$, resp.~
$v<u$:
\begin{align*}
J(u<v<w)&:=[\redbrv{u}{v},\fvsl w]_{q_{uv,w}} - [\fvsl u,\redbrv{v}{w}]_{q_{u,vw}}\\
        &\qquad\qquad\qquad+ q_{u,v} \fvsl v[\fvsl u,\fvsl w] - q_{v,w} [\fvsl u,\fvsl w]\fvsl v,\\
L(u,u< v)&:=  \bigl[\underbrace{\fvsl u,\ldots[\fvsl u}_{N_u-1},\redbrv{u}{v}]_{q_{u,u} 
q_{u,v}}\ldots\bigr]_{q_{u,u}^{N_u-1} q_{u,v}}\!\! - [\redhv{u},\fvsl v]_{q_{u,v}^{N_u}},\\
L(u,u\le v)&:=\begin{cases}
             L(u,u< v), &\text{if }u<v,\\
             L(u):=-[\redhv{u},\fvsl u]_{1}, &\text{if }u=v,
              \end{cases}\\
L(u,v<u)&:=\bigl[\ldots[\redbrv{v}{u},\underbrace{\fvsl u]_{q_{v,u}q_{u,u}}\ldots,\fvsl u}_{N_u-1}\bigr]_{q_{v,u}q_{u,u}^{N_u-1}} \!\!- [\fvsl v,\redhv{u}]_{q_{v,u}^{N_u}}.
\end{align*}

\begin{rem}
Note that 
$$
J(u<v<w)\in \bigl([\fvsl u,\fvsl v] -\redbrv{u}{v},[\fvsl v,\fvsl w] -\redbrv{v}{w}\bigr)
$$ 
by the $q$-Jacobi identity of Proposition \ref{PropqCommut}, and 
$$
L(u,u\le v)\in\bigl([\fvsl u,\fvsl v] -\redbrv{u}{v},\ \fvsl{u}^{N_u} - \redhv{u}\bigr),\quad L(u,v<u)\in\bigl([\fvsl v,\fvsl u] -\redbrv{v}{u},\ \fvsl{u}^{N_u} - \redhv{u}\bigr)
$$ 
by the restricted $q$-Leibniz formula of Proposition \ref{PropqCommut}.
\end{rem}

\subsection{The PBW criterion}



\begin{thm}\label{ThmPBWCrit} Let $L\subset\Ly$ be Shirshov closed and $I$ be an ideal of $\k\X\# \k[\G]$ as in Section \ref{SectIdealOfCharHA}. Then the following assertions are equivalent:

\begin{enumerate}
  \item[\emph{(1)}] The residue classes of $[u_1]^{r_1}[u_2]^{r_2}\ldots [u_t]^{r_t}g$ with $t\in\mathbb N$, $u_i\in L$, $u_1>\ldots>u_t$, $0< r_i<N_{u_i}$, $g\in\G$, form a $\k$-basis of the quotient algebra $(\k\X\# \k[\G])/I$.
  \item[\emph{(2)}] The algebra $\kLv \# \k[\G]$ respects the following conditions:\\ 
    \emph{(a)} \emph{$q$-Jacobi condition:} $\forall$ $u,v,w\in L$, $u<v<w$:
\begin{align*}
 J(u<v<w)\in I_{\prec\fvsl u\fvsl v\fvsl w}.
\end{align*}
  \emph{(b)} \emph{restricted $q$-Leibniz conditions:} $\forall$ $u,v\in L$ with $N_u<\infty$, 
$u\le v$ resp.~$v<u$:

           \hspace{0.6cm}\emph{\ (i)} $L(u,u\le v)\in I_{\prec\fvsl{u}^{N_u}\fvsl v},$ resp.

          \hspace{0.6cm}\emph{(ii)} $L(u,v<u)\in I_{\prec\fvsl v\fvsl{u}^{N_u}},$

\item[\emph{(2')}] The algebra $\kLv \# \k[\G]$ respects the following conditions:\\ 
    \emph{(a)} Condition \emph{(2a)} 
only for $uv\notin L$ or $\nSh{uv}{u}{v}$.\\
     \emph{(b)} \emph{(i)} Condition \emph{(2bi)} only for $u=v$ 
and $u<v$ 
where $v
\neq uv'$ 
for all $v'\in L$.

 \hspace{0.6cm}\emph{(ii)} Condition \emph{(2bii)} only for $v<u$ where $v
\neq v'u$ for all
$v'\in L$.
\end{enumerate}
\end{thm}
We need to formulate several statements over the next sections. Afterwards the proof of Theorem \ref{ThmPBWCrit} 
will be carried out in Section \ref{SectProofPBWCrit}.

\section{$(\kX \# H)/I$ as a quotient of a free algebra}\label{SectIdentFreeAlg}
In order to make the diamond lemma applicable for $(\kX \# H)/I$, also not just for the regular letters $X$ but for some super letters $[L]$, we will define a quotient of a certain free algebra, which is the special case of the following general construction: 

In this section let $X, S$ be arbitrary sets such that $X\subset S$, and $H$ be a bialgebra with $\k$-basis $G$. Then 
$$
\kX\subset \k\la S\ra\quad\text{and}\quad H=\kspan G
\subset \k\la G\ra,
$$
if we view the set $G$ as variables. Further we set $\la S,G\ra :=\la S\cup G\ra$ where we may assume that the union is disjoint. 
By omitting $\o$
$$
\kX\o H = \kspan \{u g\,|\, u\in\X, g\in G\}\subset\k\la S, G\ra
$$

Now let $\kX$  be a  $H$-module algebra. Next we define the ideals corresponding to the extension of the variable set $X$ to $S$, and to the smash product structure and the multiplication of $H$, and study their properties afterwards.

\begin{defn}\label{defnIdeals}
(1)  Let $A$ be an algebra, $B\subset A$ a subset. Then let $(B)_A$ denote the 
ideal generated by the set $B$. 

(2) Let $f_s\in\kX$ for all $s\in S$. Further let $1_H\in G$  and $f_{gh}:=gh\in H=\kspan G$ for all $g, h\in G$.
We then define the ideals
\begin{align*}
\idl{S}  := (& s-f_s \ |\ s\in S)_{\k\la S,G\ra},\\ 
\idl{G} :=\bigl(& gs - (g\sw1\cdot f_s) g\sw2,\  g h-f_{gh},\ 1_H- 1
\ |\   g, h\in G, s\in S\bigr)_{\k\la S,G\ra} ,
\end{align*}
where $1$ is the empty word  in $\k\la S,G\ra$.




\end{defn}

\begin{rem}   We may assume that $1_H\in G$, if $H\neq 0$: Suppose $1_H\notin G$ and write $1_H$ 
as a linear combination of $G$. Suppose all coefficients are $0$, then $1_H=0_H$ hence $H=0$; a contradiction. So there is a $g$ with non-zero coefficient and we can exchange this $g$ with $1_H$.
\end{rem}

\begin{ex}\label{ExIdentFreeAlgH=kG}
Let $H=\k[\G]$ be the group  algebra with the usual bialgebra structure $\Delta(g)=g\o g$ and $\varepsilon(g)=1$. Here $G=\G$, 
 $f_{gh}\in \G$ is just the product in the group, and 
$$
\idl\G=\bigl(  gs- (g\cdot f_s) g, \;  gh-f_{gh},\; 1_\G - 1\;|\;  g,h\in\G,\;s\in S\bigr).
$$
\end{ex}

\begin{lem}\label{PropertyOfIdealsISandIG} 
For any $g\in\G$ we have $$ g( \k\la S,G\ra)\subset \kspan\{u g\,|\, u\in\X, g\in G\} + \idl{G}.$$
\end{lem}

\begin{proof} Let $a_1 \ldots a_n\in\la S,G\ra$. 
We proceed by induction on $n$. If $n=1$ then either $a_1\in S$ or $a_1\in G$. Then either $ g a_1\in (g\sw1\cdot f_{a_1})g\sw2 +\idl{G}\subset \kspan\{u g\,|\, u\in\X, g\in G\} + \idl{G}$ or $ g a_1\in f_{g a_1}+\idl{G}\subset \kspan\{u g\,|\, u\in\X, g\in G\} + \idl{G}$. If $n>1$, then let us consider $g a_1a_2\ldots a_n$. Again either $a_1\in S$  or $a_1\in G$ and we argue for $ga_1$ as in the induction basis; then by using the induction hypothesis we achieve the desired form.
\end{proof}

\begin{prop}\label{IdentFreeAlg}
Assume the above situation. 
Then
$$
   \kX \# H \;\cong\; \k\la S,G\ra / (\idl{S}\!+\!\idl{G}) ,
$$
and for any ideal $I$ of $\k\la X\ra\# H$ also $\idl{S}\!+\! \idl{G}\!+\!I$ is an ideal of $\k\la S,G\ra$ such that
$$
 (\k\la X\ra\# H)/I  \; \cong \;\k\la S,G\ra / (\idl{S}\!+\! \idl{G}\!+\! I).
$$ 
Further we have the following special cases:
\begin{align}
H\cong \k:&  &  \kX & \; \cong \; \k\la S \ra / \idl{S},  &  \kX/I & \; \cong \; \k\la S\ra / (\idl{S}\!+\! I)       \label{IdentFreeAlgH=k}.\\
S=X:&        &  \kX\# H  & \; \cong \; \k\la X,\! G \ra /  \idl{G},       &   (\kX\# H)/I & \; \cong \; \k\la X,\! G\ra /(\idl{G}\!+\! I).  \label{IdentFreeAlgS=X}
\end{align}
\end{prop}
\begin{proof}
(1) The algebra map 
$$\k\la S,G\ra \rightarrow  \kX  \# H , \quad
 s\mapsto f_s\# 1_H,\quad g\mapsto 1_{\kX}\# g$$ is surjective and contains $\idl{S}+\idl{G}$ in its kernel; this is a direct calculation using the definitions. Hence we have a surjective algebra map on the quotient
\begin{align}\label{IsoIdentFreeAlg}
\k\la S,G\ra / (\idl{S}\!+\!\idl{G}) \longrightarrow \kX \# H.
\end{align}
In order to see that this map is bijective, we verify that a basis is mapped to a basis.

(a) The residue classes of the elements of  $\{u g\,|\, u\in\X, g\in G\}$  $\k$-generate $\k\la S,G\ra / (\idl{S}+\idl{G})$: Let $A\in \la S,G\ra$. Then either $A\in\la S\ra$ or it contains an element of $G$. In the first case $A\in \kX +\idl{S}$ by definition of $\idl{S}$, and then $A\in \kX 1_H+\idl{S}+\idl{G}$ since $1_H-1\in \idl\G$. In the other case let $A=A_1 g A_2$ with $A_1\in\la S\ra$, $g\in G$, $A_2\in\la S,G\ra$. We argue for $A_1$ like before, and  $g A_2\in \kspan\{u g\,|\, u\in\X, g\in G\} + \idl{G}$ by Lemma \ref{PropertyOfIdealsISandIG}. 

(b)  The residue classes of $\{u g\,|\, u\in\X, g\in G\}$ are mapped by Eq.~\eqref{IsoIdentFreeAlg} to the $\k$-basis $\X \# G$ of the right-hand side. Hence the residue classes are linearly independent, thus form a basis of $\k\la S,G\ra / (\idl{S}\!+\!\idl G)$.

(2) $\idl{S}+\idl\G+I$ is an ideal: Let $A\in\la S,G\ra$ and $a\in I\subset \kspan\{u g\,|\, u\in\X, g\in G\}$. Then by (1a) above $A\in \kspan\{u g\,|\, u\in\X, g\in G\} +\idl{S}+\idl{G}$, and since $I$ is an ideal of  $\kX \# H$,  we have $Aa,aA\in \idl{S}+\idl{G}+I$ by the isomorphism Eq.~\eqref{IsoIdentFreeAlg}.

Using the isomorphism theorem and part (1) we get
\begin{multline*}
 \k\la S,G\ra  / (\idl{S}\!+\!\idl{G}\!+\!I) 
    \cong\, 
    \bigl(\k\la S,G\ra /(\idl{S}\!+\!\idl{G})\bigr) \big/ \bigl((\idl{S}\!+\! \idl{G}\!+\!I)/(\idl{S}\!+\!\idl{G})\bigr) \,
    \cong\,
    (k\la X\ra\# H) / I,
\end{multline*}
where the last $\cong$ holds since $(\idl{S}\!+\! \idl{G}\!+\!I)/(\idl{S}\!+\!\idl{G})$ is mapped to $I$ by the isomorphism Eq.~\eqref{IsoIdentFreeAlg}.

(3) The special cases follow from the facts that $\idl{S}=0$ if $S=X$, and if $H\cong \k$ then  
$G=\{1_H\}$. Hence $\idl{G}=(1_H-1)$ and $\kX\cong \kX\#\k\cong \k\la S,\{1_H\}\ra/(\idl{S}+(1_H-1))\cong \k\la S\ra/\idl{S}$.  
\end{proof}

 We now return to the situation of Section \ref{SectCharHA},  and rewrite Proposition \ref{IdentFreeAlg} for the case $S=X_L$ and $H=\k[\G]$:

\begin{cor}\label{CorIdentFreeAlg} Let $L\subset\Ly$ be Shirshov closed and 
\begin{align*}
 \idl L&:= \bigl(\fvsl{u} -[\fvsl v,\fvsl w]\;|\; u\in L,\, \Sh{u}{v}{w}\bigr)_{\k\la X_L,\G\ra}\\
 \idl\G'&:=\bigl(g\fvsl{u}- \chi_u(g) \fvsl u  g,\;  g h-f_{gh},\; 1_\G-1\;|\; g,h\in\G,\,u\in L \bigr)_{\k\la X_L,\G\ra}.
\end{align*}
Then for any ideal $I$ of $\k\la X\ra\# \k[\G]$ also $\idl{L}\!+\! \idl\G'\!+\! I$ is an ideal of $\k\la X_L,\G\ra$ such that
$$
 (\k\la X\ra\# \k[\G])/I  \; \cong \;\k\la X_L,\G\ra / (\idl{L}\!+\! \idl\G'\!+\! I).
$$ 
Further we have the analog special cases of Proposition \ref{IdentFreeAlg}.
\end{cor}
\begin{proof} We apply Proposition \ref{IdentFreeAlg} to the case $S=X_L$, $H=\k[\G]$, $f_{x_u}=[u]$ for all $u\in L$. Then $\idl{X_L}=\bigl(\fvsl{u} -[u] \;|\; u\in L\bigr)_{\k\la X_L,\G\ra}$ and $\idl{\G}$ is as in Example    \ref{ExIdentFreeAlgH=kG}.
We are left to prove $\idl{L}\!+\! \idl\G'\!+\! I=\idl{X_L}\!+\! \idl\G\!+\! I$, which follows from the Lemma below.
\end{proof}

\begin{lem}\label{PropertyOfIdealIS} We have 
\begin{enumerate}
\item[\emph{(1)}] $[u] \in \fvsl{u} + \idl L$ for all $u\in L$; hence   $\idl{X_L}= \idl L$.
\item[\emph{(2)}] $
\idl\G \subset 
\idl\G'+\idl L
$
\end{enumerate}
\end{lem}
\begin{proof} (2) follows from (1), which we prove by induction on $\ell(u)$: For  $\ell(u)=1$ there is nothing to show. 
Let $\ell(u)>1$ and $\Sh{u}{v}{w}$.  Then by the induction assumption we have
\begin{align*}
[u]
 &=   [v][w]-  q_{v,w}  [w][v]
 \in (\fvsl{v} + \idl L)( \fvsl{w} +\idl L) -  q_{vw}  (\fvsl{w}+ \idl L)(\fvsl{v} +\idl L)\\
 &\subset [\fvsl v,\fvsl w] +\idl L = \fvsl{u}-(\underbrace{\fvsl{u} -[\fvsl v,\fvsl w]}_{\in \idl L}) + \idl L = \fvsl{u} + \idl L.
\end{align*}
\end{proof}



\begin{ex}\label{ExPreNicholsAlgA2}
Let $X=\{x_1,x_2\}\subset L=\{x_1,x_1x_2,x_2\}$. 
Then $\idl{L} = \bigl(\fvsl{12}-[\fvsl{1},\fvsl{2}]\bigr)$ and by Corollary \ref{CorIdentFreeAlg}
$
\k\la x_1,x_2\ra    \;\cong\; \k \bigl\la \fvsl 1,\fvsl{12},\fvsl 2 \;\big|\; \fvsl{12}=[\fvsl 1,\fvsl 2] \bigr\ra,
$ and
\begin{align*}
\k\la x_1,x_2\ra \# \k[\G]   \;\cong\;  \k\la \fvsl{1},\fvsl{12},\fvsl{2},\G\;|\;      
       \fvsl{12}&=[\fvsl{1},\fvsl{2}],\\
       g\fvsl{u}&= \chi_u(g)\fvsl u  g ,\,
      g h=f_{gh},\, 1_\G-1 ;
       \forall  u\in L, g,h\in\G \ra.
\end{align*}
\end{ex}


\section{Bergman's diamond lemma} 
\label{SectDL}
Following Bergman \cite{BergmanDL}, let $Y$ be a set, $\k\la Y\ra$ the free $\k$-algebra and $\Sigma$ an index set.
We fix a subset $\mathcal R=\{(W_{\sigma},f_{\sigma})\,|\,\sigma\in\Sigma\} \subset \la Y\ra\times \k\la Y\ra$, and define the ideal
$$
\idl{\mathcal R}:=(W_{\sigma}-f_{\sigma}\,|\,\sigma\in\Sigma)_{\k\la Y\ra}.
$$


An \emph{overlap of $\mathcal R$} is a triple $(A,B,C)$ such that there are $\sigma,\tau\in \Sigma$ and $A,B,C\in \la Y\ra\backslash\{1\}$ with $W_{\sigma}=AB$ and $W_{\tau}=BC$. In the same way an \emph{inclusion of $\mathcal R$} is a triple $(A,B,C)$ such that there are $\sigma\neq\tau\in \Sigma$ and $A,B,C\in \la Y\ra$ with $W_{\sigma}=B$ and $W_{\tau}=ABC$.

Let $\orddleq$ be a \emph{with $\mathcal R$ compatible well-founded monoid partial ordering} of the free monoid $\la Y\ra$, i.e.: 
\begin{itemize}
 \item $(\la Y\ra,\orddleq)$ is a partial ordered set.
\item  $B\orddl B'\Rightarrow ABC\orddl AB'C$ for all $A,B,B',C\in \la Y\ra$.
\item  Each non-empty subset of $\la Y\ra$ has a minimal element w.r.t. $\preceq_{\diamond}$.
\item $f_{\sigma}$ is a linear combination of monomials $\prec_{\diamond}W_{\sigma}$ for all $\sigma\in \Sigma$; in 
 this case we write $f_{\sigma}\prec_{\diamond}W_{\sigma}$.
\end{itemize}



For any $A\in \la Y\ra$ let $\idlDL{A}$ denote the subspace of $\k\la Y\ra$ spanned by all elements $B(W_{\sigma}-f_{\sigma})C$ with $B,C\in \la Y\ra$ such that $BW_{\sigma}C\orddl A$. The next theorem is a short version of the diamond lemma:

\begin{thm}\emph{\cite[Thm 1.2]{BergmanDL}}\label{ThmDL} Let $\mathcal R=\{(W_{\sigma},f_{\sigma})\,|\,\sigma\in\Sigma\}\subset \la Y\ra\times \k\la Y\ra$ and  $\preceq_{\diamond}$ be a with $\mathcal R$ compatible  well-founded monoid partial ordering on $\la Y\ra$. 
Then the following conditions are equivalent:
\begin{enumerate}
\item[\emph{(1)}] 
  \begin{enumerate}
    \item[\emph{(a)}] $f_{\sigma}C-Af_{\tau}\in \idlDL{ABC}$ for all overlaps $(A,B,C)$.
    \item[\emph{(b)}] $Af_{\sigma}C-f_{\tau}\in \idlDL{ABC}$ for all inclusions $(A,B,C)$.
  \end{enumerate}
\item[\emph{(2)}] The residue classes of the elements of $\la Y\ra$ which do not contain any $W_{\sigma}$  with $\sigma\in\Sigma$ as a subword form a $\k$-basis of $\k\la Y\ra / I_{\mathcal R}$.
\end{enumerate}
\end{thm}

We now define the ordering for our situation, where $L\subset\Ly$ is Shirshov closed and $Y=X_L\cup \G$: 
Let $\pi_L:\la X_L, \G\ra \rightarrow \Lv$ be the monoid map with $\fvsl u\mapsto\fvsl u$ and $g\mapsto 1$ for all $u\in L$, $g\in\G$ ($\pi_L$ deletes all $g$ in a word of $\la X_L, \G\ra$).


Moreover, for a $A\in\la X_L, \G\ra$ let $n_{\G}(A)$ denote the number of letters $g\in\G$ in the word $A$ and $t(A)$ the $n_{\G}(A)$-tuple of non-negative integers
\begin{multline*}
(\mbox{number of letters after the last $g\in\G$ in $A$},\ldots,\\
\ldots,\mbox{number 
of letters after the first $g\in\G$ in $A$})\in\N^{n_{\G}(A)}.
\end{multline*}






\begin{defn}\label{DefnOrdDL}  For $A,B\in\la X_L, \G\ra$ we define $A\prec_{\diamond} B$ by
\begin{itemize}
\item $\pi_L(A)\prec \pi_L(B)$, or 
\item $\pi_L(A)=\pi_L(B)$ and $n_{\G}(A)< n_{\G}(B)$, or
\item $\pi_L(A)=\pi_L(B)$, $n_{\G}(A)= n_{\G}(B)$ and $t(A)<t(B)$ under the lexicographical 
order of $\N^{n_{\G}(A)}$, i.e., $t(A)\neq t(B)$, and the first non-zero term of $t(B)-t(A)$ is 
positive.
\end{itemize}
\end{defn}

$\preceq_{\diamond}$ is a well-founded monoid partial ordering of $\la X_L, \G\ra$, which is straightforward to verify, and will be compatible with the later regarded $\mathcal R$.

 
Note that we have the following correspondence between $\prec$ of Section \ref{SectWellFoundOrder} and $\orddl$, which follows from the definitions: For any $U,V\in\SupWL$, $g,h\in\G$ we have $\rho(U)g, \rho(V)h\in \la X_L\ra \G$ and
\begin{align}\label{VergleichPrecVS.PrecDiamLem}
U\prec V  \; \Longleftrightarrow\; \rho(U)g\orddl \rho(V)h.
\end{align}

\section{Proof of Theorem \ref{ThmPBWCrit}}\label{SectProofPBWCrit}
Again suppose the assumptions of Theorem \ref{ThmPBWCrit}. By Corollary \ref{CorIdentFreeAlg}
$$
(\k\X\# \k[\G])/I\cong \k\la X_L, \G\ra/ (\idl{L}+\idl\G'+ I),
$$ 
thus $(\k\X\# \k[\G])/I$ has the basis 
$
[u_1]^{r_1}[u_2]^{r_2}\ldots [u_t]^{r_t}g
$
if and only if $\k\la X_L, \G\ra/ (\idl{L}+\idl\G'+ I)$ has  the basis 
$
x_{u_1}^{r_1}x_{u_2}^{r_2}\ldots x_{u_t}^{r_t}g
$
($t\in\mathbb N$, $u_i\in L$, $u_1>\ldots>u_t$, $0< r_i<N_u$, $g\in\G$).
The latter we can reformulate equivalently in terms of the Diamond Lemma \ref{ThmDL}: 

\noindent \textbullet \ We define $\mathcal R$ as the set of the elements 
\begin{align}
                                 (1_\G,\; 1            ),\label{Red1G} & \\
  ( g h  ,\; f_{gh}   ),\label{RedG} & \text{ for all } g, h\in \G ,         \\
    \bigl( g \fvsl u,\; \chi_u(g)\fvsl u g \bigr),&\text{ for all }  g\in\G, u\in L  ,   \label{RedGu}\\
         \bigl(\fvsl u\fvsl v,\;\redbrv{u}{v}+q_{u,v}\fvsl v\fvsl u \bigl), &\text{ for all } u,v\in L\text{ with }u<v ,\label{Reduv}\\
   \bigl(\fvsl{u}^{N_u},\; \redhv{u}\bigr),&\text{ for all } u\in L\text{ with }N_u<\infty \label{Reduh},
\end{align}
where we again see $\redbrv{u}{v}$, $\redhv{u}\in \kLv\o \k[\G]\subset 
\kspan \{Ug\ |\ U\in\Lv,g\in\G \}\subset \k\la X_L, \G\ra$. 
Then the residue classes of $\redbrv{u}{v},\redhv{u}$ 
modulo $\idl L+\idl\G'$ correspond to $\redbr{u}{v}$ and $\redh{u}$ by the isomorphism of Corollary \ref{CorIdentFreeAlg}, and we have $I_{\mathcal R}=\idl{L}\!+\!\idl\G'\!+\! I$.

\noindent \textbullet \ Note that $\prec_{\diamond}$ is compatible with $\mathcal R$: In Eq.~\eqref{Red1G} resp.~\eqref{RedG} we have $1\orddl 1_\G$ resp.~$f_{gh} \orddl  g h$ since $n_{\G}(1)=0<1=n_{\G}(1_\G)$ resp.~$n_{\G}(f_{gh})=1<2=n_{\G}( g h)$ ($f_{gh}\in\G$). Eq.~\eqref{RedGu}: $t(\fvsl u g)=(0)<(1)=t( g\fvsl u)$, hence $\fvsl u g \orddl  g\fvsl u$. Moreover, by \cite[Lem.~3.6]{Helbig-Presentation} we have $\redbrv{u}{v}+q_{u,v}\fvsl v\fvsl u\orddl \fvsl u\fvsl v$, 
and $\redhv{u}\orddl \fvsl{u}^{N_u}$ by assumption.

\noindent \textbullet \ By the Diamond Lemma \ref{ThmDL} we have to consider all possible overlaps and inclusions of $\mathcal R$. The only inclusions happen with Eq.~\eqref{Red1G}, namely $(1,1_\G,h)$, $(g,1_\G,1)$, $(1,1_\G,\fvsl{u})$. But they all fulfill the condition (1b) of the Diamond Lemma \ref{ThmDL}: for example $ h-f_{1_\G h}=h-h=0\in \idlDL{1_\G h}$, and $\fvsl u - \chi_u(1_\G)\fvsl u 1_\G=\fvsl u(1_\G- 1)\in \idlDL{1_\G \fvsl{u}}$. 

So we are left to check the conditon (1a) for all overlaps: $( g, h,k)$ with $g,h,k\in\G$ fulfills it by the associativity of $\G$; for $( g, h,\fvsl u)$ we have
\begin{align*}
& f_{gh} \fvsl u - \chi_u(h) g \fvsl u  h = \chi_{u}(gh)\fvsl u f_{gh}-\chi_u(h)\chi_u(g) \fvsl u g  h=0,
\end{align*}
calculating modulo $\idlDL{ g h\fvsl u}$ and using $\chi_u(f_{gh})=\chi_{u}(gh)$ since $f_{gh}\in\G$.
The next overlap is $(g,\fvsl u,\fvsl v)$ where $u<v$: Calculating modulo $\idlDL{ g\fvsl u\fvsl v}$ we get
\begin{multline*}
\chi_u(g)\fvsl u g \fvsl v -  g \bigl(\redbrv{u}{v} +q_{u,v}\fvsl v\fvsl u\bigr)
=\chi_u(g)\chi_v(g)\fvsl u\fvsl v g -\\ \chi_{uv}(g)\bigl(\redbrv{u}{v} +q_{u,v}\fvsl v\fvsl u\bigr) g
=\chi_{uv}(g)\bigl(\fvsl u\fvsl v - \bigl(\redbrv{u}{v} +q_{u,v}\fvsl v\fvsl u\bigr)\bigr) g=0,
\end{multline*}
since $\redbr{u}{v}\in (\kX\# \k[\G])^{\chi_{uv}}$ and $\fvsl u\fvsl v g\orddl  g\fvsl u\fvsl v$. For the overlap $( g,\fvsl u,\fvsl{u}^{N_u-1})$ we obtain modulo $\idlDL{ g\fvsl{u}^{N_u}}$
\begin{align*}
\chi_u(g)\fvsl u g \fvsl{u}^{N_u-1} -  g \redhv{u} = 
\chi_u(g)^{N_u}\bigl(\fvsl{u}^{N_u}-\redhv{u}\bigr) g=0,
\end{align*}
because $\redh{u} \in (\kX\# \k[\G])^{\chi_u^{N_u}}$ and $\fvsl{u}^{N_u}\fv g\orddl \fv g\fvsl{u}^{N_u}$.
The remaining overlaps are those with Eqs.~\eqref{Reduv} and \eqref{Reduh}; for these we formulate the following three Lemmata which are equivalent to (2) of the Theorem \ref{ThmPBWCrit}:
\begin{lem}
The overlap $(\fvsl u,\fvsl v,\fvsl w)$, $u<v<w$,  fulfills condition \ref{ThmDL}(1a), i.e.,
$
a:=\bigl(\redbrv{u}{v} + q_{u,v} \fvsl v\fvsl u \bigr)\fvsl w - \fvsl u\bigl(\redbrv{v}{w} +q_{v,w}\fvsl w\fvsl v\bigr) \in\idlDL{\fvsl u\fvsl v\fvsl w},
$
if and only if
$J(u<v<w)  \in \idlDL{\fvsl{u}\fvsl{v}\fvsl{w}}.
$
\end{lem}
\begin{proof} We calculate in $\k\la X_L,\G\ra$
\begin{align*}
J(u<v<w)   &= \redbrv{u}{v}\fvsl w-q_{uv,w}\fvsl w\redbrv{u}{v}-  \bigl(\fvsl u\redbrv{v}{w}-q_{u,vw}\redbrv{v}{w}\fvsl u\bigr)\\
    &\quad + q_{u,v} \fvsl v\bigl(\fvsl u\fvsl w-q_{u,w}\fvsl w\fvsl u\bigr) - q_{v,w}\bigl(\fvsl u\fvsl w-q_{u,w}\fvsl w\fvsl u\bigr)\fvsl v,\\
a   &= \redbrv{u}{v}\fvsl w + q_{u,v}\fvsl v\fvsl u\fvsl w -\fvsl u\redbrv{v}{w} -q_{v,w}\fvsl u\fvsl w\fvsl v,
\end{align*}
and show that the difference is zero modulo $\idlDL{\fvsl{u}\fvsl{v}\fvsl{w}}$:
\begin{align*}
J(u<v<w)-a &= q_{uv,w} \fvsl w\bigl(\fvsl u\fvsl v-\redbrv{u}{v}\bigr) +q_{u,vw}\bigl(\redbrv{v}{w}-\fvsl v\fvsl w\bigr)\fvsl u \\
&=q_{uv,w} \fvsl w\bigl(q_{u,v}\fvsl v\fvsl u\bigr) - q_{u,vw}\bigl(q_{v,w}\fvsl w\fvsl v\bigr)\fvsl u=0.
\end{align*}
since $\fvsl w\fvsl u\fvsl v,\fvsl v\fvsl w\fvsl u\orddl\fvsl u\fvsl v\fvsl w$.
\end{proof}

\begin{lem}
The overlaps $\bigl(\fvsl{u}^{N_u-1},\fvsl{u},\fvsl{v}\bigr)$ resp.~$\bigl(\fvsl{u},\fvsl{v},\fvsl{v}^{N_v-1}\bigr)$  fulfill condition \ref{ThmDL}(1a), i.e.,
$
\redhv{u}\fvsl{v} -\fvsl{u}^{N_u-1}\bigl(\redbrv{u}{v} +q_{u,v}\fvsl{v}\fvsl{u}\bigr) \in\idlDL{\fvsl{u}^{N_u}\fvsl{v}}
$
resp.~
$
\bigl( \redbrv{u}{v} +q_{uv}\fvsl{v}\fvsl{u}\bigr)\fvsl{v}^{N_v-1}-\fvsl{u}\redhv{v} \in\idlDL{\fvsl{u}\fvsl{v}^{N_v}}
$
if and only if 
$L(u,u<v) \in \idlDL{\fvsl{u}^{N_u}\fvsl v}$ resp.~$L(u,u>v)\in \idlDL{\fvsl v\fvsl{u}^{N_u}}$.
\end{lem}

\begin{proof} We prove it for $\bigl(\fvsl{u}^{N_u-1},\fvsl{u},\fvsl{v}\bigr)$; the other overlap is proved analogously.
We set $r:=N_u-1$, then $\ord \, q_{u,u}=r+1$. Using the $q$-Leibniz formula of Proposition \ref{PropqCommut} we get
\begin{align*}
&\fvsl{u}^r\big( \redbrv{u}{v} + q_{u,v}\fvsl{v}\fvsl{u}\bigr)-\redhv{u}\fvsl{v}=\\
&=\bigl[ \fvsl{u}^r, \redbrv{u}{v} \bigr]_{q_{u,u}^rq_{u,v}}+q_{u,u}^rq_{u,v}\redbrv{u}{v}\fvsl{u}^r\\
&\qquad\qquad\qquad\qquad+q_{u,v}\bigl[ \fvsl{u}^r, \fvsl{v}\bigr]_{q_{u,v}^r}\fvsl{u} +q_{u,v}^{r+1}  \fvsl{v}\fvsl{u}^{r+1}-\redhv{u}\fvsl{v}\\
&= \sum_{i=0}^{r}q_{u,u}^iq_{u,v}^i\tbinom{r}{i}_{q_{u,u}}\bigl[\underbrace{\fvsl{u},\ldots [\fvsl{u}}_{r-i},\redbrv{u}{v} ]_{q_{u,u}q_{u,v}}\ldots\bigr]_{q_{u,u}^{r-i}q_{u,v}} \fvsl{u}^i\\
&\quad +
\sum_{i=0}^{r-1}q_{u,v}^{i+1}\tbinom{r}{i}_{q_{u,u}}\bigl[\underbrace{\fvsl{u},\ldots[\fvsl{u}}_{r-i},\fvsl{v} ]_{q_{u,v}}\ldots\bigr]_{q_{u,u}^{r-i-1}q_{u,v}} \fvsl{u}^{i+1} + q_{u,v}^{r+1}\fvsl{v}\fvsl{u}^{r+1}-\redhv{u}\fvsl{v}.
\end{align*}
Because of $\fvsl{u}^{r-i}\fvsl{v}\fvsl{u}^{i+1}\orddl \fvsl{u}^{r+1}\fvsl{v}$ for all $0\le i\le r$, this is modulo $\idlDL{\fvsl{u}^{r+1}\fvsl{v}}$ equal to 
\begin{align*}
&\sum_{i=0}^{r}q_{u,u}^iq_{u,v}^i\tbinom{r}{i}_{q_{u,u}}\bigl[\underbrace{\fvsl{u},\ldots [\fvsl{u}}_{r-i},\redbrv{u}{v} ]_{q_{u,u}q_{u,v}}\ldots\bigr]_{q_{u,u}^{r-i}q_{u,v}} \fvsl{u}^i\\
&\quad+
\sum_{i=0}^{r-1}q_{u,v}^{i+1}\tbinom{r}{i}_{q_{u,u}}\bigl[\underbrace{\fvsl{u},\ldots[\fvsl{u}}_{r-i-1},\redbrv{u}{v}]_{q_{u,u}q_{u,v}}\ldots\bigr]_{q_{u,u}^{r-i-1}q_{u,v}} \fvsl{u}^{i+1} - \bigl[\redhv{u},\fvsl{v}\bigr]_{q_{u,v}^{r+1}}.
\end{align*}
Now shifting the index of the second sum, we obtain
\begin{align*}
&\bigl[\underbrace{\fvsl{u},\ldots[\fvsl{u}}_{r},\redbrv{u}{v}]_{q_{u,u}q_{u,v}}\ldots\bigr]_{q_{u,u}^{r}q_{u,v}}- \bigl[\redhv{u},\fvsl{v}\bigr]_{q_{u,v}^{r+1}}\\
&\quad+
\sum_{i=1}^{r}q_{u,v}^i \Bigl(q_{u,u}^i\tbinom{r}{i}_{q_{u,u}}+\tbinom{r}{i-1}_{q_{u,u}}\Bigr)
\bigl[\underbrace{\fvsl{u},\ldots[\fvsl{u}}_{r-i},\redbrv{u}{v}]_{q_{u,u}q_{u,v}}\ldots\bigr]_{q_{u,u}^{r-i}q_{u,v}} \fvsl{u}^i.
\end{align*}
Finally we obtain the claim, since $q_{u,u}^i\tbinom{r}{i}_{q_{u,u}}+\tbinom{r}{i-1}_{q_{u,u}}=\tbinom{r+1}{i}_{q_{u,u}}=0$ for all $1\le i\le r$, by Eq.~\eqref{PascalDreieck} and $\ord\, q_{u,u}=r+1$.
\end{proof}

\begin{lem}
The overlaps $\bigl(\fvsl{u}^{N_u-i},\fvsl{u}^i,\fvsl{u}^{N_u-i}\bigr)$ fulfill condition \ref{ThmDL}(1a) for all $1\le i< N_u$,
if and only if the overlap $\bigl(\fvsl{u}^{N_u-1},\fvsl{u},\fvsl{u}^{N_u-1}\bigr)$ fulfills condition \ref{ThmDL}(1a), if and only if $
L(u)
\in \idlDL{\fvsl{u}^{N_u+1}}.
$
\end{lem}
\begin{proof}This is evident. \end{proof}

\noindent \textbullet \ We are left to prove the equivalence of (2) to its weaker version (2') of Theorem \ref{ThmPBWCrit}: For (2'a) we show that if $uv\in L$ and $\Sh{uv}{u}{v}$, then conditon (2a) is already fulfilled: By definition  $\redbrv{u}{v}=\fvsl{uv}$  and  
$$\bigl[ \redbrv{u}{v}, \fvsl{w} \bigr]_{q_{uv,w}}=
\bigl[ \fvsl{uv}, \fvsl{w} \bigr]  = \redbrv{uv}{w}
$$ 
modulo $I_{\prec\fvsl{u}\fvsl{v}\fvsl{w}}$. Now certainly  $\nSh{uvw}{uv}{w}$, thus 
$$\redbrv{uv}{w} = \partial_{u}^{\rho}(\redbrv{v}{w}) +q_{v,w}\redbrv{u}{w}\fvsl{v} -q_{u,v}\fvsl{v}\redbrv{u}{w}
$$ 
by Eq.~\eqref{RedCommutDefnVariableNotSh}. Hence in this case the $q$-Jacobi condition is fulfilled by the $q$-derivation formula of Proposition \ref{PropqCommut}.

For (2'b) of Theorem \ref{ThmPBWCrit}  it is enough to show the following: Let condition (2bi) hold for $u=v$, i.e., $[\fvsl u,  \redhv{u}]_{1}\in I_{\prec\fvsl{u}^{N_u+1}}$. Then, if condition 
(2bi) holds for some $u<v$ with $N_u<\infty$, then (2bi) also holds for $u<uv$ (whenever $uv\in L$). 
Analogously, if  (2bii) holds for $v<u$ with $N_u<\infty$, then also (2bii) holds for $vu<u$ (whenever $vu\in L$). 

Note that if $u<v$, then $uv<v$: Either $v$ does not begin with $u$, then $uv<v$; or let $v=uw$ for some $w\in\X$. Then $u<v=uw<w$ since $v\in\Ly$. Hence $uv=uuw<uw=v$. 

We will prove the first part (2'bi), (2'bii) is the same argument.  But before we formulate the following
\begin{lem}\label{LemSupplAssert} Let $a\in \kLv\# \k[\G]$, $A,W\in \Lv$ such that $a\preceq_L A\prec W$. 
Then $a \in I_{\prec W}$ if and only if $a\in I_{\preceq A}$.

\end{lem}
\begin{proof} Clearly $I_{\preceq A}\subset I_{\prec W}$, since $A\prec W$. 
So denote by $\{(W_\sigma,f_\sigma)\ |\ \sigma\in\Sigma\}$ the set of Eqs.~\eqref{Reduv} and \eqref{Reduh} with $f_\sigma\prec_L W_\sigma$, and let $a \in I_{\prec W}$, i.e., $a$ is a linear combination of $Ug(W_\sigma-f_\sigma)Vh$  with $U,V\in \Lv$ such that $UW_\sigma V\prec W$. Denote by $E$ the $\prec$-biggest word of all $UW_\sigma V$ with non-zero coefficient. $E\succ A$ contradicts the assumption  $a\preceq_L A\prec W$. Hence $E\preceq A$ and therefore $f\in I_{\preceq A}$.
\end{proof}

Suppose (2bi) for $u<v$ with $N_u<\infty$ and $uv\in L$, i.e.,
\begin{align*}
&\bigl[\underbrace{\fvsl u,\ldots[\fvsl u}_{N_u-1},\fvsl{uv}]_{q_{u,u}q_{u,v}}\ldots\bigr]_{q_{u,u}^{N_u-1} q_{u,v}} - [\redhv{u},\fvsl v]_{q_{u,v}^{N_u}} \in I_{\prec\fvsl{u}^{N_u}\fvsl v}\\
\Leftrightarrow 
&\bigl[\underbrace{\fvsl u,\ldots[\fvsl u}_{N_u-2},\redbrv{u}{uv}]_{q_{u,u}^2 q_{u,v}}\ldots\bigr]_{q_{u,u}^{N_u-1} q_{u,v}} - [\redhv{u},\fvsl v]_{q_{u,v}^{N_u}} \in I_{\preceq \fvsl{u}^{N_u-1}\fvsl{w}U\fvsl{v}},
\end{align*}
for some $w\in L$ with $w>u$ and $U\in \Lv$ such that $\ell(U)+\ell(w)=\ell(u)$.
Here we used the relation $[\fvsl u,\fvsl{uv}]_{q_{u,uv}}-\redbrv{u}{uv}$, and Lemma \ref{LemSupplAssert} since the above polynomial is $\preceq \fvsl{u}^{N_u-1}\fvsl{w}U\fvsl{v}$ (by assumption $\redbr{u}{uv}\preceq_L [uuv]$, $\redh{u}\prec_L [u]^{N_u}$). Hence the condition (2bi) for $u<uv$ reads
\begin{align*}
&\bigl[\underbrace{\fvsl u,\ldots[\fvsl u}_{N_u-1},\redbrv{u}{uv}]_{q_{u,u}^2 q_{u,v}}\ldots\bigr]_{q_{u,u}^{N_u} q_{u,v}} - [\redhv{u},\fvsl{uv}]_{q_{u,u}^{N_u}q_{u,v}^{N_u}} \in I_{\prec\fvsl{u}^{N_u}\fvsl{uv}}\\
\Leftrightarrow&
\bigl[\fvsl u,  [\redhv{u},\fvsl v]_{q_{u,v}^{N_u}}  \bigr]_{q_{u,u}^{N_u} q_{u,v}} - [\redhv{u},\fvsl{uv}]_{q_{u,u}^{N_u}q_{u,v}^{N_u}} \in I_{\prec\fvsl{u}^{N_u}\fvsl{uv}},
\end{align*}
since $\fvsl u I_{\preceq \fvsl{u}^{N_u-1}\fvsl{w}U\fvsl{v}}, I_{\preceq \fvsl{u}^{N_u-1}\fvsl{w}U\fvsl{v}}\fvsl u\subset I_{\prec\fvsl{u}^{N_u}\fvsl{uv}} $ ($w>u$ and $w$ cannot begin with $u$ since $\ell(w)\le \ell(u)$, hence $w>uv$. By the $q$-Jacobi identity 
\begin{align*}
 \bigl[\fvsl u,  [\redhv{u},\fvsl v]_{q_{u,v}^{N_u}}  \bigr]_{q_{u,u}^{N_u} q_{u,v}}
&=
\bigl[[\fvsl u,  \redhv{u}]_{q_{u,u}^{N_u}},\fvsl v \bigr]_{q_{u,v}^{N_u +1}} +q_{u,u}^{N_u}\redhv{u}\fvslbr{u}{v}-q_{u,v}^{N_{u}}\fvslbr{u}{v}\redhv{u}\\
&= 
\bigl[[\fvsl u,  \redhv{u}]_{1},\fvsl v \bigr]_{q_{u,v}^{N_u+1}} +[\redhv{u},\fvsl{uv}]_{q_{u,v}^{N_{u}}} =[\redhv{u},\fvsl{uv}]_{q_{u,v}^{N_{u}}}.
\end{align*}
For the last two ``='' we used $q_{u,u}^{N_{u}}=1$, the relation $\fvslbr{u}{v}-\fvsl{uv}$ and $[\fvsl u,  \redhv{u}]_{1}\in I_{\prec\fvsl{u}^{N_u+1}}$ (We can use this condition:  Note that $[\fvsl u,  \redhv{u}]_{1}\preceq \fvsl u^{N_u}\fvsl{w'}U'$ for some $w'\in L$, $w'>u$, $U'\in \Lv$, $\ell(U')+\ell(w')=\ell(u)$, hence  $[\fvsl u,  \redhv{u}]_{1}\in I_{\preceq \fvsl u^{N_u}\fvsl{w'}U'}$ by Lemma \ref{LemSupplAssert}. Therefore $\fvsl v I_{\preceq \fvsl u^{N_u}\fvsl{w'}U'},I_{\preceq \fvsl u^{N_u}\fvsl{w'}U'}\fvsl v\subset I_{\prec\fvsl{u}^{N_u}\fvsl{uv}}$, like before).

\section{PBW basis in rank one} \label{SectExamplesRankOne}
We want to apply the  PBW basis  criterion to Hopf algebras of rank one and two for some fixed $L\subset\Ly$.
Especially we want to treat liftings of Nichols algebras. Therefore we define the following scalars which will guarantee a $\Gh$-graduation:
\begin{defn}\label{DefnLiftCoeffic} Let $L\subset\Ly$. 
Then we define coefficients
 $\mu_u\in\k$ for all $u\in D(L)$, and $\lambda_{w}\in \k$ for all $w\in C(L)$ by 
\begin{align*}
\mu_u=0,\text{ if } g_u^{N_u}=1\text{ or }\chi_u^{N_u}\neq \varepsilon, \qquad\quad
\lambda_{w}=0,\text{ if } g_{w}=1 \text{ or }\chi_{w}\neq \varepsilon,
\end{align*}
and otherwise they can be chosen arbitrarily.
\end{defn}

In this section let $V$ be a 1-dimensional vector space with basis $x_1$ and $\ord q_{11}=N\le\infty$. Since 
$
T(V)\cong \k[x_1]$ we have $\Ly=\{x_1\}$. We give the condition when
$
(T(V)\# \k[\G]) /\bigl( x_1^N-\redh{1} \bigr)
$
has the PBW basis $\{x_1\}$. By the PBW Criterion \ref{ThmPBWCrit} the only condition in $\k[ \fvsl 1] \#\k[\G]$ is 
$$
[\redhv{1},\fvsl 1]_{1}\in I_{\prec\fvsl{1}^{N+1}}.
$$

\begin{exs}Let $\operatorname{char} \k=0$ and $q\in\k^{\times}$ with $\ord q=N\ge 2$.
\begin{enumerate}
\item \emph{Nichols algebra $A_1$}. $T(V)/\bigl(x_1^N\bigr)$ has basis $\{x_1^{r}\,|\, 0\le r<N \}$.

\item \emph{Taft Hopf algebra}. Let $\Z/(N)=\langle g_1\rangle$ and $\chi_1(g_1):=q$.  The set $\{ x_1^{r}g \;|\;  0\le r< N,g\in \Z/(N)\}$ is a basis of
$
 T(q)\ \cong\  \bigl(\k[x_1]\#\k[\Z/(N)]\bigr)\big/(x_1^N).
$ 
 \item \emph{Radford Hopf algebra}. Let $\Z/(N^2)=\langle g_1\rangle$ and $\chi_1(g_1):=q$. The set $\{ x_1^{r}g \;|\;  0\le r< N,g\in \Z/(N^2)\}$ is a basis of
$
 r(q)\cong (\k[x_1]\#\k[\Z/(N^2)])/(x_1^N-(1-g_1^N)).
$ 
\item \emph{Liftings $A_1$}. The set $\{x_1^{r}g\,|\, 0\le r<N,g\in\G \}$ is a basis of
$
(T(V)\#\k[\G])/\bigl(x_1^{N}-\mu_1(1-g_1^{N})\bigr),
$
\end{enumerate}
\end{exs}
\begin{proof} 
(1) and (2) clearly fulfill the only condition above, since $\redh{1}=0$.

(3) is a special case of (4): It is $d_1\in(\kX\# \k[\G])^{\chi_1^N}$ by  Definition \ref{DefnLiftCoeffic} of $\mu_1$. Further 
$$
\bigl[ \mu_1(1-g_1^{N})  ,\fvsl{1} \bigr]_{1}=\mu_1\bigl[ 1  ,\fvsl{1} \bigr]_{1}- \mu_1\bigl[ g_1^{N} ,\fvsl{1} \bigr]_{1}= - \mu_1 (q_{11}^{N}-1) \fvsl{1}g_1^{N}=0,
$$ since $\ord q_{11}=N$. 
\end{proof}

\section{PBW basis in rank two and redundant relations}\label{SectExamplesRankTwo}
Let $V$ be a 2-dimensional vector space with basis $x_1$, $x_2$, hence $T(V)\cong \k\la x_1,x_2\ra$. In this chapter we apply the PBW Criterion \ref{ThmPBWCrit} to verify  for certain $L\subset\Ly$ that the algebra
$$
(T(V)\#\k[\G])/ I,
$$
with $I$ as in Section \ref{SectIdealOfCharHA}, has the PBW basis $[L]$. In particular, we examine
 the Nichols algebras and their liftings of \cite{Helbig-Lift}. Moreover, we will see how to find the redundant relations, and in addition, we will treat some classical examples. 

\subsection{PBW basis for $L=\{x_1 < x_2\}$}\label{SectPBWBasisA1xA1}
This is the easiest case and covers the Cartan Type $A_1\times A_1$, as well as many other examples. We are interested when $[L]$ builds up a PBW Basis of 
$$
(T(V)\#\k[\G])/\bigl([x_1x_2]-\red{12},\;x_1^{N_1}-\redh{1},\; x_2^{N_2}-\redh{2}\bigr),
$$
with $N_1=\ord q_{11},N_2=\ord q_{22}\in\{2,3,\ldots,\infty\}$. 
If $N_1=N_2=\infty$, then by the PBW Criterion \ref{ThmPBWCrit} 
 there is no condition in $\k\la\fvsl 1,\fvsl 2\ra\#\k[\G]$ 
such that we can choose $\red{12}$ arbitrarily with $\red{12} \prec_L [x_1x_2]$ and $\deg_{\Gh}(\red{12})=\chi_{1}\chi_2$:

\begin{exs}\label{ExLiftA1Basis}$ $
\begin{enumerate}
 \item \emph{Quantum plane}. The set $\{x_2^{r_2} x_1^{r_1}\;|\; r_2,r_1\ge 0\}$ is a basis of 
$
Q(q_{12}) \cong T(V)/([x_1x_2]).
$

\item \emph{Weyl algebra}.  If $q_{12}=1$, then $\{x_2^{r_2} x_1^{r_1}\;|\; r_2,r_1\ge 0\}$ is a basis of
$W\cong T(V)/([x_1x_2]-1).$
\end{enumerate}
\end{exs}

If $\ord q_{11}=N_{1}< \infty$ or $\ord q_{22}=N_2<\infty$, then by the PBW Criterion \ref{ThmPBWCrit} we have to check 
\begin{align}
&\label{ExA1timesA1AssocPowers}\bigl[\redhv{1},\fvsl{1} \bigr]_{1}\in  I_{\prec\fvsl{1}^{N_1+1}}, \quad\text{or}\quad \bigl[\redhv{2},\fvsl{2}\bigr]_{1}\in I_{\prec\fvsl{2}^{N_2+1}},\text{ and}\\
\label{ExA1timesA1Leibniz}
&\bigl[\underbrace{\fvsl{1},\ldots\bigl[\fvsl{1}}_{N_1-1},\redv{12}\bigr]_{q_{11}q_{12}}\ldots\bigr]_{q_{11}^{N_1-1}q_{12}}  -  \bigl[\redhv{1},\fvsl{2} \bigr]_{q_{12}^{N_1}} \in I_{\prec\fvsl{1}^{N_1}\fvsl 2},\text{ or}\\
\label{ExA1timesA1Leibniz2}&\bigl[\ldots
\bigl[\redv{12},\underbrace{\fvsl{2}\bigr]_{q_{12}q_{22}}\ldots,\fvsl{2}}_{N_2-1}
\bigr]_{q_{12}q_{22}^{N_2-1}}
-
\bigl[\fvsl{1},\redhv{2}\bigr]_{q_{12}^{N_2}} \in I_{\prec \fvsl{1}\fvsl{2}^{N_2}}.
\end{align}

\begin{exs}  Let $\lambda_{12},\mu_{1},\mu_2\in \k$ as in Definition \ref{DefnLiftCoeffic}.
\begin{enumerate}
\item \emph{Nichols algebra $A_1\times A_1$}. Let $q_{12}q_{21}=1$, then $\{x_2^{r_2} x_1^{r_1} \;|\;  0\le r_i< N_i \}$ is a basis of
$$
T(V)/\bigl([x_1x_2],\, x_1^{N_1},\, x_2^{N_2}\bigr).
$$
 \item \emph{Liftings $A_1\times A_1$}. Let $q_{12}q_{21}=1$, then $\{x_2^{r_2} x_1^{r_1} g \;|\;  0\le r_i< N_i,g\in\G \}$ is a basis of
$$
(T(V)\#\k[\G])/\bigl([x_1x_2]-\lambda_{12} (1-g_{12}),\;x_1^{N_1}-\mu_1(1-g_1^{N_1}),\; x_2^{N_2}-\mu_2(1-g_2^{N_2})\bigr).
$$

\item \emph{Book Hopf algebra}. Let $q\in\k^{\times}$ with $\ord q=N>2$, $\Z/(N)= \langle g_1 \rangle$, $g:=g_2:=g_2$, and $\chi_1(g_i):=q^{-1}$, $\chi_2(g_i):=q$ for $i=1,2$. Then $\{x_2^{r_2} x_1^{r_1} g \;|\;  0\le r_i< N,g\in\G \}$ is a basis of
$
h(1,q)\cong \bigl(\k\la x_1,x_2\ra\#\k[\Z/(N)]\bigr)/\bigl([x_1x_2],\ x_1^{N},\ x_2^{N}\bigr).
$

\item \emph{Frobenius-Lusztig kernel}. Let $q\in\k^{\times}$ with $\ord q=N>2$, $\Z/(N)=\langle g_1\rangle$, $g:=g_2:=g_1$, and $\chi_1(g_i):=q^{-2}$, $\chi_2(g_i):=q^2$ for $i=1,2$. Then  $\{x_2^{r_2} x_1^{r_1} g \;|\;  0\le r_i< N,g\in\G \}$ is a basis of
 $
u_q(\mathfrak{sl}_2)\cong  \bigl(\k\la x_1,x_2\ra\#\k[\Z/(N)]\bigr)/\bigl([x_1x_2]-(1-g^2),\ x_1^{N},\ x_2^{N}\bigr).
$
\end{enumerate}
\end{exs}
\begin{proof}
In (1) it is $\redh{1} =\redh{2}= \red{12}=0$. (3) and (4) are special cases of (2): By definition of $\lambda_{12},\mu_1,\mu_2$ the elements have the required $\Gh$-degree. As in Example \ref{ExLiftA1Basis} we show 
conditions Eq.~\eqref{ExA1timesA1AssocPowers}. 
Eq.~\eqref{ExA1timesA1Leibniz}: We have $\chi_1\chi_2=\varepsilon$ if $\lambda_{12}\neq 0$, hence $q_{11}q_{12}=1$ and then $q_{11}=q_{11}q_{12}q_{21}=q_{21}$, since $q_{12}q_{21}=1$.
 Using these equations we calculate
\begin{multline*}
 \bigl[\underbrace{\fvsl{1},\ldots\bigl[\fvsl{1}}_{N_1-1},\lambda_{12} (1-g_1g_2)\bigr]_{q_{11}q_{12}}\ldots\bigr]_{q_{11}^{N_1-1}q_{12}}
=-\lambda_{12}(1-q_{11}^2)\ldots(1-q_{11}^{N_1})\fvsl{1}^{N_1-1} g_1g_2 =0.
\end{multline*}
Further $\chi_i^{N_i}= \varepsilon$ if $\mu_1\neq 0$, thus $q_{21}^{N_1}=1$; by taking $q_{12}q_{21}=1$ to the $N_1$-th power, we deduce $q_{12}^{N_1}=1$. Then 
$
\bigl[\mu_1(1-g_{1}^{N_1}),\fvsl 2\bigr]_{q_{12}^{N_1}}=\mu_1(1-q_{12}^{N_1})\fvsl 2=0.$
 The remaining condition Eq.~\eqref{ExA1timesA1Leibniz2} works in a similar way.
\end{proof}

\subsection{PBW basis for $L=\{x_1 < x_1x_2 < x_2\}$} \label{SectPBWbasisT2}
We now examine the case when $[L]$ is a PBW Basis of $(T(V)\#\k[\G])/I$, where $I$ is generated by the following elements
\begin{align*}
[x_1x_1x_2]&-\red{112},  &x_1^{N_1}&-\redh{1},\\
                         [x_1x_2x_2]&-\red{122},  &  [x_1x_2]^{N_{12}}&-\redh{12},\\
                                       &           &     x_2^{N_2}&-\redh{2},
\end{align*}
with $\ord q_{11}=N_1$, $\ord q_{12,12}=N_{12}$, $\ord q_{22}=N_2\in \{2,3,\ldots,\infty\}$. We have in $\k\la \fvsl{1},\fvsl{12},\fvsl{2}\ra\# \k[\G]$ the elements
$$
\redbrv{1}{12}=\redv{112},\quad  \redbrv{1}{2}=\fvsl{12},\quad  \redbrv{12}{2}=\redv{122}.
$$
At first we want to study the conditions in general. By Theorem \ref{ThmPBWCrit}(2') we have to check the following: The only Jacobi condition is for $1<12<2$, namely
\begin{align}\label{ExA2JacobiCond}
\bigl[\redv{112},\fvsl{2}\bigr]_{q_{112,2}} - \bigl[\fvsl{1},\redv{122}\bigr]_{q_{1,122}} 
  + (q_{1,12}  - q_{12,2}) \fvsl{12}^2 \in  I_{\prec\fvsl 1\fvsl{12}\fvsl 2}.
\end{align}
There are the following restricted $q$-Leibniz conditions:
If $N_1<\infty$, then we have to check Eqs.~\eqref{ExA1timesA1AssocPowers} and \eqref{ExA1timesA1Leibniz} for $1<2$; 
note that we can omit the restricted Leibniz condition for  $1<12$ in (2') of Theorem \ref{ThmPBWCrit}.
In the same way if $N_2<\infty$, then there are the conditions Eqs.~\eqref{ExA1timesA1AssocPowers} and \eqref{ExA1timesA1Leibniz2}  
for $1<2$;  we can omit the condition for $12<2$.
Further  Eq.~\eqref{ExA1timesA1Leibniz} resp. \eqref{ExA1timesA1Leibniz2} is equivalent to
\begin{align}\label{ExA2LeibnizN1}
&\bigl[\underbrace{\fvsl 1,\ldots[\fvsl 1}_{N_1-2},\redv{112}]_{q_{11}^2
q_{12}}\ldots\bigr]_{q_{11}^{N_1-1} q_{12}} - [\redhv{1},\fvsl 2]_{q_{12}^{N_1}} \in I_{\prec\fvsl{1}^{N_1}\fvsl 2},\\
\label{ExA2LeibnizN2}
&\bigl[\ldots[\redv{122},\underbrace{\fvsl 2]_{q_{12}q_{22}^2}\ldots,\fvsl 2}_{N_2-2}\bigr]_{q_{12}q_{22}^{N_2-1}} 
- [\fvsl 1,\redhv{2}]_{q_{12}^{N_2}} \in I_{\prec\fvsl 1\fvsl{2}^{N_2}}.
\end{align}
In the case  $N_1=2$ resp.~$N_2=2$ then  condition Eq.~\eqref{ExA2LeibnizN1}  resp.~\eqref{ExA2LeibnizN2} is
$$
\redv{112}  - [\redhv{1},\fvsl 2]_{q_{12}^{2}} \in I_{\prec\fvsl{1}^{2}\fvsl 2}\quad \text{ resp. }\quad
\redv{122}  - [\fvsl 1,\redhv{2}]_{q_{12}^{2}} \in I_{\prec\fvsl 1\fvsl{2}^{2}}.
$$
Here we see with Corollary \ref{CorIdentFreeAlg} that by the restricted $q$-Leibniz formula $[x_1x_1x_2]-\red{112}\in (x_1^2-\redh{1})$ resp.~$[x_1x_2x_2]-\red{122}\in (x_2^2-\redh{2})$, hence these two relations are redundant. Suppose $[\redh{1},x_2]_{q_{12}^{2}}\prec_L [x_1x_1x_2]$ resp.~$[x_1,\redh{2}]_{q_{12}^{2}}\prec_L [x_1x_2x_2]$. Thus if we define
\begin{align}\label{ExA2LeibnizN1=2orN2=2}
\redv{112}:= [\redhv{1},\fvsl 2]_{q_{12}^{2}}\quad\text{ resp. }\quad\redv{122}:= [\fvsl 1,\redhv{2}]_{q_{12}^{2}},
\end{align}
then condition Eq.~\eqref{ExA2LeibnizN1}  resp.~\eqref{ExA2LeibnizN2} is fulfilled.

Finally, if $N_{12}<\infty$, then there are the conditions 
\begin{align}\label{ExA2LeibnizN12}
\begin{split}\bigl[\redhv{12},\fvsl{12}\bigr]_{1}\in  I_{\prec\fvsl{12}^{N_{12}+1}},\\
\bigl[\ldots[\redv{112},\underbrace{\fvsl{12}]_{q_{1,12}q_{12,12}}\ldots,\fvsl{12}}_{N_{12}-1}\bigr]_{q_{1,12}q_{12,12}^{N_{12}-1}} 
- [\fvsl 1,\redhv{12}]_{q_{1,12}^{N_{12}}} \in I_{\prec\fvsl 1\fvsl{12}^{N_{12}}},\\
\bigl[\underbrace{\fvsl{12},\ldots[\fvsl{12}}_{N_{12}-1},\redv{122}]_{q_{{12},{12}} 
q_{{12},2}}\ldots\bigr]_{q_{{12},{12}}^{N_{12}-1} q_{{12},2}} - [\redhv{{12}},\fvsl 2]_{q_{{12},2}^{N_{12}}} \in I_{\prec\fvsl{{12}}^{N_{12}}\fvsl 2}.
\end{split}
\end{align}


Now we want to take a closer look at Eq.~\eqref{ExA2JacobiCond}. Essentially, there are two cases: If $q_{11}=q_{22}$ we set $q:=q_{112,2}= q_{1,122}$ and then Eq.~\eqref{ExA2JacobiCond} reads
\begin{align}\label{JacCondq11=q22}
\bigl[\redv{112},\fvsl{2}\bigr]_{q} - \bigl[\fvsl{1},\redv{122}\bigr]_{q} \in  I_{\prec\fvsl 1\fvsl{12}\fvsl 2}.
\end{align}
Else if $q_{11}\neq q_{22}$. Suppose $
N_{12}=\ord q_{12,12}=2$, then we define
\begin{align*}
\redh{12}:=-(q_{1,12}-q_{12,2})^{-1}\bigl(\bigl[\red{112},x_2\bigr]_{q_{1,2}q_{12,2}} - \bigl[\fvsl{1},\red{122}\bigr]_{q_{1,122}}  \bigl).
\end{align*} 
It is $[x_1x_2]^2-\redh{12}\in \bigl([x_1x_1x_2]-\red{112},\, [x_1x_2x_2]-\red{122}\bigr)$ by the $q$-Jacobi identity, see Eq.~\eqref{ExA2JacobiCond} and Corollary \ref{CorIdentFreeAlg}, i.e., this relation is redundant. Further  $\redh{12}\in (\kX\# \k[\G]))^{\chi_{12}^2}$. Let us assume
that $\redh{12}\prec_L [x_1x_2]^2$, e.g., $\red{122},\red{112}$ are linear combinations of monomials of length $<3$.
Then for 
\begin{align}
\label{ExJacobiq11neqq22}
\redhv{12} := -(q_{1,12}-q_{12,2})^{-1}\bigl(\bigl[\redv{112},\fvsl{2}\bigr]_{q_{1,2}q_{12,2}} -\bigl[\fvsl{1},\redv{122}\bigr]_{q_{1,122}} \bigl)
\end{align}
condition Eq.~\eqref{ExA2JacobiCond} is fulfilled.

As a demonstration we want to proof that the Hopf algebras coming from liftings of a Nichols algebra with Cartan matrix $A_2$ \cite[Thm.~5.9]{Helbig-Lift}, admit a PBW basis $[L]$ (this is already known for liftings of Nichols algebras of Cartan type $A_2$  \cite{AS-A2}, but not for non-Cartan type):

\begin{prop}[Liftings $A_2$] \label{ExLiftingsA2}\label{PropPBWLiftingsA2} Consider the Hopf algebras $(\k\la x_1,x_2\ra\# \k[\G])/ I$ where $I$ depends upon $(q_{ij})$ as follows:
\begin{enumerate}
\item[\emph{(1)}] Cartan type $A_2$: $q_{12}q_{21}=q_{11}^{-1}=q_{22}^{-1}$.\\
\emph{(a)}  If $q_{11}=-1$, then let $I$ be generated by
\begin{align*}
\qquad\quad 
x_1^{2}-\mu_1(1-g_1^2),\qquad  
    [x_1x_2]^{2} - 4\mu_1 q_{21}x_2^2-\mu_{12}(1-g_{12}^2), 
    \qquad  
    x_2^{2}-\mu_2(1-g_2^2).
\end{align*}

\emph{(b)} 
If $\ord q_{11}=3$, then let $I$ be generated by
\begin{align*}
[x_1x_1x_2]&-\lambda_{112}(1-g_{112}),\quad
   [x_1x_2x_2]-\lambda_{122}(1-g_{122}),\\
    x_1^{3}&-\mu_1(1-g_1^3),\\ 
[x_1x_2]^{3}& +(1-q_{11})q_{11}\lambda_{112}[x_1x_2x_2] \\
            &\qquad-    \mu_1(1-q_{11})^3 x_2^3  -\mu_{12}(1-g_{12}^3),\\
x_2^{3}&-\mu_2(1-g_2^3).
\end{align*}

\emph{(c)} If $N:=\ord q_{11}\ge 4$, 
then  then let $I$ be generated by, see \emph{\cite{AS-A2}},
\begin{align*}
[x_1x_1x_2]&,\quad [x_1x_2x_2],\\ 
                 x_1^{N}&-\mu_1(1-g_1^N),\\ 
                  [x_1x_2]^{N}&-\mu_1(q_{11}-1)^N q_{21}^{\frac{N(N-1)}{2}} x_2^N-\mu_{12}(1-g_{12}^N),\\ 
                  x_2^{N}&-\mu_2(1-g_2^N).
\end{align*}

\item[\emph{(2)}] Let $q_{12}q_{21}=q_{11}^{-1}$, $q_{22}=-1$.\\
 \emph{(a)} If $4 \neq N:=\ord q_{11} \ge 3$, then let $I$ be generated by
\begin{align*}
[x_1x_1x_2], \qquad  x_1^{N}-\mu_1(1-g_1^N),\qquad
  x_2^2  -\mu_2(1-g_2^2).
\end{align*} 

\emph{(b)} If $\ord q_{11}=4$, then let $I$ be generated by
\begin{align*}
[x_1x_1x_2]-\lambda_{112}(1-g_{112}), \qquad x_1^{4}-\mu_1(1-g_1^4),\qquad  x_2^2-\mu_2(1-g_2^2).
\end{align*}

\item[\emph{(3)}] Let $q_{11}=-1$, $q_{12}q_{21}=q_{22}^{-1}$.\\
 \emph{(a)} If  $4 \neq N:= \ord q_{22}\ge 3$, then let $I$ be generated by 
\begin{align*}
[x_1x_2x_2], \qquad x_1^2  -\mu_1(1-g_1^2) ,\qquad
  x_2^{N}-\mu_2(1-g_2^N).
\end{align*} 

\emph{(b)} If $\ord q_{22}=4$, then let $I$ be generated by
\begin{align*}
[x_1x_2x_2]-\lambda_{122}(1-g_{122}), \qquad  x_1^2-\mu_1(1-g_1^2),\qquad  x_2^{4}-\mu_2(1-g_2^4).
\end{align*} 

\item[\emph{(4)}]   Let $q_{11}=q_{22}=-1$ and $ N:=\ord q_{12}q_{21}\ge 3$.\\
\emph{(a)} If $q_{12}\neq \pm 1$, then let $I$ be generated by
\begin{align*}
 x_1^{2}-\mu_1(1-g_1^2),\qquad 
[x_1x_2]^{N}&-\mu_{12}(1-g_{12}^N),\qquad
x_2^{2}.\end{align*}
\emph{(b)} If $q_{12}=\pm 1$, then let $I$ be generated by
\begin{align*} 
 x_1^{2},\qquad
[x_1x_2]^{N}-\mu_{12}(1-g_{12}^N),\qquad
x_2^{2}-\mu_2(1-g_2^2).
\end{align*}
\end{enumerate}
All of these Hopf algebras have basis $\{x_2^{r_2}[x_1x_2]^{r_{12}} x_1^{r_1}g\;|\; 0\le r_u<N_u\text{ for all }u\in L,\,g\in\G\}.$
\end{prop}

\begin{proof} Note that all defined ideals are $\Gh$-homogeneous by the definition of the coefficients. The conditions Eq.~\eqref{ExA1timesA1AssocPowers} are exactly as in Example \ref{ExLiftA1Basis}.

(1a) We have $N_1=N_2=N_{12}=2$. Since  $\redhv{1}=\mu_1(1-g_1^2)$ we have by the argument preceding Eq.~\eqref{ExA2LeibnizN1=2orN2=2}, that necessarily 
$$
\red{112}=[\mu_1(1-g_1^2),x_2]_{q_{12}^{2}} \quad\text{and}\quad \red{122}=[x_1,\mu_2(1-g_2^2)]_{q_{12}^{2}}
$$
and the conditions Eqs.~\eqref{ExA2LeibnizN1} and \eqref{ExA2LeibnizN2} are fulfilled. Note that $\red{112}=\mu_1 (1-q_{12}^{2}) x_2=0$: either $\mu_1=0$ or else $\mu_1\neq 0$, but then $\chi_1^2=\varepsilon$ and $q_{21}^2=1$. By squaring the assumption $q_{12}q_{21}=-1$, we obtain $q_{12}^2=1$. In the same way $\red{122}=0$.

Then the conditions Eq.~\eqref{ExA2LeibnizN12} are 
\begin{align*}\bigl[4\mu_1 q_{21}\fvsl{2}^2+\mu_{12}(1-g_{12}^2),\fvsl{12}\bigr]_{1}&\in I_{\prec\fvsl{12}^3}\\
[0,\fvsl{12}]_{q_{1,12}q_{12,12}}- [\fvsl 1,4\mu_1 q_{21}\fvsl{2}^2+\mu_{12}(1-g_{12}^2)]_{q_{1,12}^{2}}& \in I_{\prec\fvsl 1\fvsl{12}^{2}},\\
[\fvsl{12},0]_{q_{{12},{12}}q_{{12},2}} - [4\mu_1 q_{21}\fvsl{2}^2+\mu_{12}(1-g_{12}^2),\fvsl 2]_{q_{{12},2}^{2}} &\in I_{\prec\fvsl{{12}}^{2}\fvsl 2}.
\end{align*}
Again, if $\mu_1\neq 0$, then $q_{12}^2=q_{21}^2=1$, hence $q_{1,12}^{2}=1$ and $q_{2,12}^2=1$. If $\mu_{12}\neq 0$, then $\chi_{12}^2=\varepsilon$ and $q_{1,12}^{2}=1$; in this case also $q_{12}^2=q_{21}^2=1$. 
Thus modulo $I_{\prec\fvsl{12}^3}$ we have 
\begin{multline*}
 \bigl[4\mu_1 q_{21}\fvsl{2}^2+\mu_{12}(1-g_{12}^2),\fvsl{12}\bigr]_{1}
= 4\mu_1 q_{21}  \bigl[\fvsl{2}^2,\fvsl{12}\bigr]_{1}- \mu_{12}(q_{12,12}^2-1)\fvsl{12}g_{12}^2\\
= 4\mu_1\mu_2 q_{21}  \bigl[1-g_2^2,\fvsl{12}\bigr]_{1}
=-4\mu_1\mu_2 q_{21}  (q_{2,12}^2-1)\fvsl{12}g_2^2=0.
\end{multline*}
Further modulo  $I_{\prec\fvsl 1\fvsl{12}^{2}}$ we get
\begin{align*}
 [\fvsl 1,4\mu_1 q_{21}\fvsl{2}^2 +\mu_{12}(1-g_{12}^2)]_{1}
&= 4\mu_1q_{21}[\fvsl 1,\fvsl{2}^2]_{1} + \mu_{12}[\fvsl 1, 1-g_{12}^2]_{1}\\
&= 4\mu_1q_{21}\redv{122} -\mu_{12}(1-q_{12,1}^2)\fvsl 1 g_{12}^2=0,
\end{align*}
which means that the second condition is fulfilled. The third one of Eq.~\eqref{ExA2LeibnizN12} works analogously.

The last condition is Eq.~\eqref{ExA2JacobiCond}, or equivalently condition Eq.~\eqref{JacCondq11=q22}  since $q_{11}=q_{22}$:
$$
\bigl[0,\fvsl{2}\bigr]_{q} - \bigl[\fvsl{1},0\bigr]_{q} =0 \in  I_{\prec\fvsl 1\fvsl{12}\fvsl 2}.
$$

(1b) 
Either $\lambda_{112}=\lambda_{122}=0$, 
or $\chi_{112}= \varepsilon$ and/or $\chi_{122}= \varepsilon$, from where we conclude $q:=q_{11}=q_{12}=q_{21}=q_{22}$.
We start with Eq.~\eqref{ExA2JacobiCond}: Since $q^3=1$ we have
$
\bigl[\lambda_{112}(1-g_{112}),\fvsl{2}\bigr]_{1} - \bigl[\fvsl{1},\lambda_{122}(1-g_{122})\bigr]_{1}=0.
$ 
We continue with Eq.~\eqref{ExA2LeibnizN1}: Either $\mu_1=0$ or $\chi_1^3=\varepsilon$, hence  $q_{21}^3=1$ and then also $q_{12}^3=(q_{12}q_{21})^3=q_{11}^{-3}=1$. Then 
 $
\bigl[\fvsl 1,  \lambda_{112}(1-g_{112})  \bigr]_{1} - [\mu_1(1-g_1^3),\fvsl 2]_{1} =0.$ 
Next, Eq.~\eqref{ExA2LeibnizN2}: In the same way, $\mu_2\neq 0$ or $q_{21}^3=q_{12}^3=1$. Then 
$
\bigl[\lambda_{122}(1-g_{122}),\fvsl 2\bigr]_{1} - [\fvsl 1,\mu_2(1-g_2^3)]_{1}=0.
$ 
For Eq.~\eqref{ExA2LeibnizN12} we have $q_{1,12}^3=1$ if $\mu_{12}\neq 0$. Thus $q_{12}^3=1$, moreover $q_{21}^3=(q_{12}q_{21})^3=q_{11}^{-3}=1$.
Hence modulo $I_{\prec\fvsl 1\fvsl{12}^{3}}$ we have
\begin{multline*}
 \bigl[[\lambda_{112}(1-g_{112}),\fvsl{12}]_{q_{1,12}q_{12,12}},\fvsl{12}\bigr]_{q_{1,12}q_{12,12}^{2}} \\
- \bigl[\fvsl 1,-(1-q_{11})q_{11}\lambda_{112} \lambda_{122}(1-g_{122})   +   \mu_1(1-q_{11})^3 \fvsl{2}^3  +\mu_{12}(1-g_{12}^3)\bigr]_{q_{1,12}^{3}}=0,
\end{multline*}
since each summand is zero. Further a straightforward calculation shows
\begin{multline*}
 \bigl[\fvsl{12},[\fvsl{12},\lambda_{122}(1-g_{122})]_{q_{{12},{12}}q_{{12},2}}\bigr]_{q_{{12},{12}}^{2} q_{{12},2}}\\ - \bigl[-(1-q_{11})q_{11}\lambda_{112} \lambda_{122}(1-g_{122})   +   \mu_1(1-q_{11})^3 \fvsl{2}^3  +\mu_{12}(1-g_{12}^3),\fvsl 2\bigr]_{q_{{12},2}^{2}}=0.
\end{multline*}
Finally, an easy calculation shows that $$\bigl[-(1-q_{11})q_{11}\lambda_{112} \lambda_{122}(1-g_{122})+ \mu_1(1-q_{11})^3 \fvsl{2}^3  +\mu_{12}(1-g_{12}^3),\fvsl{12}\bigr]_{1}=0$$ modulo $I_{\prec\fvsl{12}^{4}}$, again by definition of the coefficients.

(1c) is a generalization of (1a) (and (1b) if $\lambda_{112}=\lambda_{122}=0$) and works completely in the same way (only the Serre-relations $[x_1x_1x_2]=[x_1x_2x_2]=0$ are not redundant, as they are (1a)). We leave this to the reader.

(2a) We leave this to the reader and prove the little more complicated (2b):  
 Since we have $N_2=2$, as in (1a) we deduce from  Eq.~\eqref{ExA2LeibnizN1=2orN2=2}, that 
$
 \red{122}=[x_1,\mu_2(1-g_2^2)]_{q_{12}^{2}}=\mu_2(q_{21}^2-1)x_1g_2^2
$
and the condition Eq.~\eqref{ExA2LeibnizN2} is fulfilled.

If $\lambda_{112}\neq 0$ then $q_{11}=q_{21}$ of order $4$, $q_{12}=q_{22}=-1$; if $\mu_1\neq 0$ then $q_{12}^4=1$. Then Eq.~\eqref{ExA2LeibnizN1} is fulfilled:
$
 \bigl[\fvsl 1,[\fvsl 1,\lambda_{112}(1-g_{112})]_{1}\bigr]_{q_{11}} - [\mu_1(1-g_1^4),\fvsl 2]_{1} =0,
$
since both summands are zero.

It is $q_{11}\neq q_{22}$, $\ord q_{12,12}=2$ and $\redv{112}$ resp.~$\redv{122}$ are linear combinations of monomials of length $0$ resp.~$1$.  By the discussion before Eq.~\eqref{ExJacobiq11neqq22}, we see that $[x_1x_2]^2-\redh{12}$ is redundant and for 
\begin{align*}
\redhv{12}&:=-(q_{1,12}-q_{12,2})^{-1}\bigl(\bigl[\lambda_{112}(1-g_{112}),x_2\bigr]_{-1} - \bigl[x_1, \mu_2(q_{21}^2-1)x_1g_2^2\bigr]_{q_{11}} \bigl)\\
&= -q_{12}^{-1}(q_{11}+1)^{-1}\bigl(\lambda_{112}2  x_2 - \mu_2\underbrace{(q_{21}^2-1)(1-q_{11}q_{21}^2)}_{=:q}x_1^2g_2^2 \bigl)
\end{align*}
the condition Eq.~\eqref{ExA2JacobiCond} is fulfilled. We are left to show the conditions Eq.~\eqref{ExA2LeibnizN12} $\bigl[\redhv{12},\fvsl{12}\bigr]_{1}\in  I_{\prec\fvsl{12}^{3}},$
\begin{align*}
\bigl[ \redv{112}, \fvsl{12}\bigr]_{q_{112,12}}-\bigl[\fvsl{1},   \redhv{12}    \bigr]_{q_{1,12}^2} \in I_{\prec \fvsl 1\fvsl{12}^2}\quad\text{and}\quad \bigl[ \fvsl{12},\redv{122} \bigr]_{q_{12,122}}-\bigl[\redhv{12},\fvsl{2} \bigr]_{q_{12,2}^2} \in I_{\prec\fvsl{12}^2\fvsl{2}}.
\end{align*}
We calculate the first one: Modulo $I_{\prec\fvsl{12}^{3}}$ we get
\begin{align*}\bigl[\redhv{12}&,\fvsl{12}\bigr]_{1}=
-q_{12}^{-1}(q_{11}+1)^{-1}\bigl(-\lambda_{112}2\underbrace{\bigl[\fvsl{12},x_2\bigr]_{1}}_{=\redv{122}}-\mu_2q
\underbrace{\bigl[x_1^2g_2^2,\fvsl{12}\bigr]_{1}}_{=q_{21}^2[\fvsl{1}^2,\fvsl{12}]_{q_{1,12}^2}g_2^2}\bigr).
\end{align*}
Now by the $q$-derivation property $[\fvsl{1}^2,\fvsl{12}]_{q_{1,12}^2}= \fvsl 1\redv{112}+q_{1,12}\redv{112}\fvsl 1=\lambda_{112}(1-q_{11})\fvsl 1.$ Because of the coefficient $\lambda_{112}$ the two summands in the parentheses have the coefficient $\pm 4\lambda_{112}\mu_2$, hence cancel. (3) works exactly as (2).

(4a) Since we have $N_1=N_2=2$, as in (1a) we deduce from  Eq.~\eqref{ExA2LeibnizN1=2orN2=2}, that 
$$
\red{112}=[\mu_1(1-g_1^2),x_2]_{q_{12}^{2}}=\mu_1(1-q_{12}^2)x_2 \quad\text{and}\quad \red{122}=[x_1,0]_{q_{12}^{2}}=0
$$
and the conditions Eqs.~\eqref{ExA2LeibnizN1} and \eqref{ExA2LeibnizN2} are fulfilled.

For the second condition of Eq.~\eqref{ExA2LeibnizN12} one can easily show by induction
\begin{multline*}
\bigl[\ldots[\redv{112},\underbrace{\fvsl{12}]_{q_{1,12}q_{12,12}}\ldots,\fvsl{12}}_{N-1}\bigr]_{q_{1,12}q_{12,12}^{N-1}}\\
= \mu_1(1-q_{12}^2)  \bigl[\ldots[\fvsl 2,\underbrace{\fvsl{12}]_{q_{11}q_{12}^2= q_{21}}\ldots,\fvsl{12}}_{N-1}\bigr]_{q_{11}q_{12}^{N}q_{21}^{N-1}} = \mu_1 \prod_{i=0}^{N-1}(1-q_{12}^{i+2}q_{21}^i)\fvsl 2\fvsl{12}^{N-1}=0.
\end{multline*}
The last equation holds since for $i=N-2$ we have $1-q_{12}^Nq_{21}^{N-2}=0$: if $\mu_1\neq 0$ then $q_{21}^2=1$ and $(q_{12}q_{21})^N=q_{12,12}^N=1$.
Further also $[\fvsl 1,\redhv{12}]_{q_{1,12}^{N}}=
[\fvsl 1,\mu_{12}(1-g_{12}^N)]_{1} = -\mu_{12}(1-q_{12,1}^N)\fvsl 1 g_{12}^N=0$, since either $\mu_{12}= 0$ or $q_{12}^N=q_{21}^N=(-1)^N$ such that $q_{12,1}^N=(-1)^N(-1)^N=1$. This proves the second  condition of Eq.~\eqref{ExA2LeibnizN12}. The third of Eq.~\eqref{ExA2LeibnizN12} is easy since $\red{122}=0$, and the first of Eq.~\eqref{ExA2LeibnizN12} is a direct computation.

Finally,  Eq.~\eqref{ExA2JacobiCond} is  Eq.~\eqref{JacCondq11=q22}, since $q_{11}=q_{22}$:
$
\bigl[\mu_1(1-q_{12}^2)\fvsl 2,\fvsl{2}\bigr]_{q_{112,2}} - \bigl[\fvsl{1},0\bigr]_{q_{1,122}}=0 
$
because of the relation $\fvsl 2^2=0$. 

(4b) works analogously to (4a). Note that here $\red{112}=0$ and $\red{122}=[x_1,\mu_2(1-g_2^2)]_{1}=\mu_2(q_{21}^2-1)x_1g_2^2$.
\end{proof}






\subsection{PBW basis for $L=\{x_1 < x_1x_1x_2 < x_1x_2 < x_2\}$}\label{SectPBWbasisT3} 

This PBW basis $[L]$ occurs in the Nichols algebras with Cartan matrix $B_2$ and their liftings 
\cite[Prop.~5.11,Thm.~5.13]{Helbig-Lift}. More generally, we list the conditions when $[L]$ is a PBW Basis of $(T(V)\#\k[\G])/I$ where $I$ is generated by
\begin{align*}
[x_1x_1x_1x_2] &-\red{1112},    &  x_1^{N_1}&-\redh{1},\\
[x_1x_1x_2x_1x_2]&-\red{11212}, &   [x_1x_1x_2]^{N_{112}}&-\redh{112},\\
[x_1x_2x_2]&-\red{122} ,         &  [x_1x_2]^{N_{12}}&-\redh{12}, \\ 
                     &           & x_2^{N_2}&-\redh{2}.
\end{align*}
In $\k\la\fvsl 1,\fvsl{112},\fvsl{12},\fvsl 2\ra\#\k[\G]$ we have the following $\redbrv{u}{v}$ ordered by $\ell(uv)$, $u,v\in L$: If $\Sh{uv}{u}{v}$ then
\begin{align*}
\redbrv{1}{2}  &= \fvsl{12},    &    \redbrv{12}{2}  &= \redv{122}  ,  &\redbrv{112}{12} &= \redv{11212},\\
\redbrv{1}{12}  &= \fvsl{112} , &    \redbrv{1}{112}  &= \redv{1112},
\end{align*}
and for $\nSh{1122}{112}{2}$ by Eq.~\eqref{RedCommutDefnVariableNotSh}
\begin{align*}
 \redbrv{112}{2} &= \partial_{1}^{\rho}(\redbrv{12}{2}) +q_{12,2}\redbrv{1}{2}\fvsl{12} -q_{1,12}\fvsl{12}\redbrv{1}{2},\\
                 &= \partial_{1}^{\rho}(\redv{122}) +(q_{12,2} -q_{1,12})\fvsl{12}^2.
\end{align*}
We have for $1 < 112 < 2$, $1 < 112 < 12$ and $112 < 12 < 2$ the following $q$-Jacobi conditions (note that we can leave out  $1<12<2$):
\begin{align}\label{ExB2JacobiCond1}
\begin{split}
&\bigl[\redv{1112},\fvsl{2}\bigr]_{q_{1112,2}} - \bigl[\fvsl{1},\redbrv{112}{2} \bigr]_{q_{1,1122}} \\
&\qquad\qquad+ q_{1,112} \fvsl{112}[\fvsl 1,\fvsl{2}] - q_{112,2} [\fvsl 1,\fvsl 2]\fvsl{112}
   \in  I_{\prec\fvsl 1\fvsl{112}\fvsl 2}\\
\Leftrightarrow& \bigl[\redv{1112},\fvsl{2}\bigr]_{q_{1112,2}} - \bigl[\fvsl{1},\partial_{1}^{\rho}(\redv{122})\bigr]_{q_{1,1122}}\\
 &\qquad - (q_{12,2} -q_{1,12})\redv{11212}- (q_{12,2} -q_{1,12}) q_{1,12}(q_{12,12}+1)\fvsl{12}\fvsl{112}\\
&\qquad\qquad + q_{1,112} \redv{11212}+ q_{112,2}(q_{1,112}q_{112,1}-1) \fvsl{12}\fvsl{112}
   \in  I_{\prec\fvsl 1\fvsl{112}\fvsl 2}\\
\Leftrightarrow& \bigl[\redv{1112},\fvsl{2}\bigr]_{q_{1112,2}} - \bigl[\fvsl{1},\partial_{1}^{\rho}(\redv{122})\bigr]_{q_{1,1122}} + \underbrace{q_{12}(q_{11}^2- q_{22} +q_{11})}_{=:q}\redv{11212}\\
 &\qquad + \underbrace{q_{12}^2(q_{22}(q_{11}^4q_{12}q_{21}-1)- q_{11}(q_{22} -q_{11}) (q_{12,12}+1))}_{=:q'}\fvsl{12}\fvsl{112}
   \in  I_{\prec\fvsl 1\fvsl{112}\fvsl 2}
\end{split}
\end{align}
If $q\neq 0$, we see that $[x_1x_1x_2x_1x_2]-\red{11212}\in  \bigl([x_1x_1x_1x_2]-\red{1112}, [x_1x_2x_2]-\red{122}\bigr)$ is redundant with 
$$
\red{11212}=-q^{-1}\bigl(\bigl[\red{1112},x_2\bigr]_{q_{1112,2}} - \bigl[x_1,\partial_{1}(\red{122})\bigr]_{q_{1,1122}} + q'[x_1x_2][x_1x_1x_2] \bigr)
$$ by Corollary \ref{CorIdentFreeAlg} and the $q$-Jacobi identity of Proposition \ref{PropqCommut}. We have $\deg_{\Gh}(\red{11212})=\chi_{11212}$; suppose that $\red{11212}\prec_L [x_1x_1x_2x_1x_2]$ (e.g. $\red{1112}$ resp.~$\red{122}$ are linear combinations of monomials of length $<4$ resp.~$<3$) then condition Eq.~\eqref{ExB2JacobiCond1} is fulfilled for
\begin{align*}
\redv{11212}:=-q^{-1}\bigl(\bigl[\redv{1112},\fvsl 2\bigr]_{q_{1112,2}} - \bigl[\fvsl 1, \partial_{1}^{\rho}(\redv{122}) \bigr]_{q_{1,1122}} + q'\fvsl{12}\fvsl{112} \bigr).
\end{align*}
There are three cases, where the coefficients $q,q'$ are of a better form for our setting: Since
$$
q=q_{12}\bigl((3)_{q_{11}}-(2)_{q_{22}}\bigr),\quad q'=q_{12}\bigl(q(1+ q_{11}^2q_{12}q_{21}q_{22})-q_{11}q_{12}(2)_{q_{22}}\bigr),
$$
we have
\begin{align*}
q&=q_{12}q_{11}\neq 0, & q'&=- q_{12}q_{11}^2q(1-q_{11}^2q_{12}q_{21}),      &&\text{if }q_{11}^2=q_{22},\\
q&=q_{12}(3)_{q_{11}}, & q'&=q_{12}q(1-q_{11}^2q_{12}q_{21}),                &&\text{if }q_{22}=-1,\\
q&=-q_{12}(2)_{q_{22}},& q'&=-q_{12}q(1+q_{11}+ q_{11}^2q_{12}q_{21}q_{22}) ,&&\text{if }\ord q_{11}=3.
\end{align*}
The second $q$-Jacobi condition for $1 < 112 < 12$ reads
\begin{align}\label{ExB2JacobiCond3}
\begin{split}
&\bigl[\redv{1112},\fvsl{12}\bigr]_{q_{1112,12}} - \bigl[\fvsl 1,\redv{11212}\bigr]_{q_{1,11212}}\\ 
&\qquad\qquad  + q_{1,112} \fvsl{112}[\fvsl 1,\fvsl{12}] - q_{112,12} [\fvsl 1,\fvsl{12}]\fvsl{112}\in I_{\prec\fvsl 1\fvsl{112}\fvsl{12}}\\
\Leftrightarrow& \bigl[\redv{1112},\fvsl{12}\bigr]_{q_{1112,12}} - \bigl[\fvsl 1,\redv{11212}\bigr]_{q_{1,11212}}
  +\underbrace{q_{11}^2q_{12}( 1 - q_{12}q_{21}q_{22})}_{=:q''} \fvsl{112}^2 \in I_{\prec\fvsl 1\fvsl{112}\fvsl{12}}
\end{split}
\end{align}
If $q''\neq 0$ then we see that $[x_1x_1x_2]^2-\redh{112}\in  \bigl([x_1x_1x_1x_2]-\red{11212}, [x_1x_1x_2x_1x_2]-\red{11212}\bigr)$ is redundant with 
$
\redh{112}=-q''^{-1}\bigl(   \bigl[\red{1112},[x_1x_2]\bigr]_{q_{1112,12}} - \bigl[x_1,\red{11212}\bigr]_{q_{1,11212}} \bigr)
$ 
by Corollary \ref{CorIdentFreeAlg} and the $q$-Jacobi identity of Proposition \ref{PropqCommut}. It is $\deg_{\Gh}(\redh{112})=\chi_{112}^2$; suppose that $\redh{112}\prec_L [x_1x_1x_2]^2$ then condition Eq.~\eqref{ExB2JacobiCond2} is fulfilled for
\begin{align*}
\redhv{112}:=-q''^{-1}\bigl(   \bigl[\redv{1112},\fvsl{12}\bigr]_{q_{1112,12}} - \bigl[x_1,\redv{11212}\bigr]_{q_{1,11212}} \bigr)
\end{align*}
If further $\ord q_{112,112}=2$ then we have to consider the restricted $q$-Leibniz  conditions for $\redhv{112}$ (see below).

The last $q$-Jacobi condition for $112 < 12 < 2$ is
\begin{align}\label{ExB2JacobiCond2}
\begin{split}
&\bigl[\redv{11212},\fvsl{2}\bigr]_{q_{11212,2}} - \bigl[\fvsl{112},\redv{122}\bigr]_{q_{112,122}} \\
 &\qquad\qquad+ q_{112,12} \fvsl{12}[\fvsl{112},\fvsl{2}] - q_{12,2} [\fvsl{112},\fvsl{2}]\fvsl{12} \in  I_{\prec\fvsl{112}\fvsl{12}\fvsl 2}\\
\Leftrightarrow&\bigl[\redv{11212},\fvsl{2}\bigr]_{q_{11212,2}} - \bigl[\fvsl{112},\redv{122}\bigr]_{q_{112,122}} \\
 &\qquad+ q_{112,12} \fvsl{12}\partial_{1}^{\rho}(\redv{122})  - q_{12,2} \partial_{1}^{\rho}(\redv{122})
\fvsl{12}\\
&\qquad\qquad
+\underbrace{q_{12}^2q_{22}(q_{22} -q_{11})(q_{11}^2q_{12}q_{21}- 1)}_{=:q'''} \fvsl{12}^3
 \in  I_{\prec\fvsl{112}\fvsl{12}\fvsl 2}\\
\end{split}
\end{align}
If $q'''\neq 0$ then we see that $[x_1x_2]^3-\redh{12}\in  \bigl([x_1x_1x_2x_1x_2]-\red{11212}, [x_1x_2x_2]-\red{122}\bigr)$ is redundant with $\redh{12}:=-q'''^{-1}\bigl(   \bigl[\red{11212},x_2\bigr]_{q_{11212,2}} - \bigl[[x_1x_1x_2],\red{122}\bigr]_{q_{112,122}} + q_{112,12} [x_1x_2]\partial_{1}(\red{122})  - q_{12,2} \partial_{1}(\red{122})[x_1x_2]\bigr)$ by Corollary \ref{CorIdentFreeAlg} and the $q$-Jacobi identity of Proposition \ref{PropqCommut}. It is $\deg_{\Gh}(\redh{12})=\chi_{12}^3$; suppose that $\redh{12}\prec_L [x_1x_1]^3$ (e.g., $\red{11212}$ resp.~$\red{122}$ are linear combinations of monomials of length $<5$ resp.~$<3$)  then condition Eq.~\eqref{ExB2JacobiCond2} is fulfilled for
\begin{align*}
\redhv{12}:&=-q''^{-1}\bigl(   \bigl[\redv{11212},\fvsl{2}\bigr]_{q_{11212,2}} - \bigl[\fvsl{112},\redv{122}\bigr]_{q_{112,122}} \\
 &\qquad\qquad+ q_{112,12} \fvsl{12}\partial_{1}^{\rho}(\redv{122})  - q_{12,2} \partial_{1}^{\rho}(\redv{122})
\fvsl{12}              \bigr)
\end{align*}
If further $\ord q_{12,12}=3$ then we have to consider the $q$-Leibniz conditions for $\redhv{12}$ (see below).

There are the following restricted $q$-Leibniz conditions:
If $N_1<\infty$, then $\bigl[\redhv{1},\fvsl{1}\bigr]_{1}\in  I_{\prec\fvsl{1}^{N_{1}+1}}$ and for $1<2$ (we can omit $1<12,1<112$)
\begin{align}\label{ExB2LeibnizN1}
\bigl[\underbrace{\fvsl 1,\ldots[\fvsl 1}_{N_1-3},\redv{1112}]_{q_{11}^3
q_{12}}\ldots\bigr]_{q_{11}^{N_1-1} q_{12}} - [\redhv{1},\fvsl 2]_{q_{12}^{N_1}} \in I_{\prec\fvsl{1}^{N_1}\fvsl 2}.
\end{align}
If $N_2<\infty$, then $\bigl[\redhv{2},\fvsl{2}\bigr]_{1}\in  I_{\prec\fvsl{2}^{N_{2}+1}}$ and for $1<2$ (we can omit $12<2,112<2$)
\begin{align}\label{ExB2LeibnizN2}
\bigl[\ldots[\redv{122}\underbrace{\fvsl 2]_{q_{12}q_{22}^2}\ldots,\fvsl 2}_{N_2-2}\bigr]_{q_{12}q_{22}^{N_2-1}} 
- [\fvsl 1,\redhv{2}]_{q_{12}^{N_2}} \in I_{\prec\fvsl 1\fvsl{2}^{N_2}}.
\end{align}
If $N_{12}<\infty$, then $\bigl[\redhv{12},\fvsl{12}\bigr]_{1}\in  I_{\prec\fvsl{12}^{N_{12}+1}}$ and for $1<12$, $12<2$ (we can omit $112<12$)
\begin{align}\label{ExB2LeibnizN12}
\begin{split}
\bigl[\ldots[\redv{112},\underbrace{\fvsl{12}]_{q_{1,12}q_{12,12}}\ldots,\fvsl{12}}_{N_{12}-1}\bigr]_{q_{1,12}q_{12,12}^{N_{12}-1}} 
- [\fvsl 1,\redhv{12}]_{q_{1,12}^{N_{12}}} \in I_{\prec\fvsl 1\fvsl{12}^{N_{12}}},\\
\bigl[\underbrace{\fvsl{12},\ldots[\fvsl{12}}_{N_{12}-1},\redv{122}]_{q_{{12},{12}} 
q_{{12},2}}\ldots\bigr]_{q_{{12},{12}}^{N_{12}-1} q_{{12},2}} - [\redhv{{12}},\fvsl 2]_{q_{{12},2}^{N_{12}}} \in I_{\prec\fvsl{{12}}^{N_{12}}\fvsl 2}.
\end{split}
\end{align}
If $N_{112}<\infty$, then $\bigl[\redhv{112},\fvsl{112}\bigr]_{1}\in  I_{\prec\fvsl{112}^{N_{112}+1}}$ and for $1<112$, $112<12$, $112<2$
\begin{align}\label{ExB2LeibnizN112}
\begin{split}
\bigl[\ldots[\redv{1112},\underbrace{\fvsl{112}]_{q_{1,112}q_{112,112}}\ldots,\fvsl{112}}_{N_{112}-1}\bigr]_{q_{1,112}q_{112,112}^{N_{112}-1}} 
- [\fvsl 1,\redhv{112}]_{q_{1,112}^{N_{112}}} &\in I_{\prec\fvsl 1\fvsl{112}^{N_{112}}}\\
\bigl[\underbrace{\fvsl{112},\ldots[\fvsl{112}}_{N_{112}-1},\redv{11212}]_{q_{112,112} 
q_{112,12}}\ldots\bigr]_{q_{112,112}^{N_{112}-1} q_{112,12}}- [\redhv{112},\fvsl{12}]_{q_{112,12}^{N_{112}}} &\in I_{\prec\fvsl{112}^{N_{112}}\fvsl{12}}\\
\bigl[\underbrace{\fvsl{112},\ldots[\fvsl{112}}_{N_{112}-1},\redbrv{112}{2}]_{q_{112,112} 
q_{112,2}}\ldots\bigr]_{q_{112,112}^{N_{112}-1} q_{112,2}} - [\redhv{112},\fvsl 2]_{q_{112,2}^{N_{112}}} &\in I_{\prec\fvsl{112}^{N_{112}}\fvsl 2}
\end{split}
\end{align}

The proof that the liftings of \cite[Thm.~5.13]{Helbig-Lift} have the PBW basis $[L]$ consists in replacing the $\redv{uv}$ and $\redhv{u}$ in the conditions above, like it was done before in Proposition \ref{PropPBWLiftingsA2}. We leave this to the reader.

\subsection{PBW basis for $L=\{x_1 < x_1x_1x_2 < x_1x_2  < x_1x_2x_2 < x_2\}$}  \label{SectPBWbasisT4} 


This PBW basis $[L]$ appears in the Nichols algebras of non-standard type and their liftings of \cite[Thm.~5.17 (1)]{Helbig-Lift}. Generally, we ask for the conditions when $[L]$ is a PBW Basis of $(T(V)\#\k[\G])/I$ where $I$ is generated by
\begin{align*}
[x_1 x_1x_1x_2] &-\red{1112},      &  x_1^{N_1}&-\redh{1},\\
[x_1x_1x_2x_2]&-\red{1122},       &  [x_1x_1x_2]^{N_{112}}&-\redh{112},\\
[x_1x_1x_2x_1x_2]&-\red{11212},    &   [x_1x_2]^{N_{12}}&-\redh{12}, \\
[x_1x_2  x_1x_2x_2]&-\red{12122},  &  [x_1x_2x_2]^{N_{122}}&-\redh{122},\\
[x_1x_2x_2x_2]&-\red{1222},        &   x_2^{N_2}&-\redh{2}.
\end{align*}
 In $\k\la\fvsl 1,\fvsl{112},\fvsl{12},\fvsl{122},\fvsl 2\ra\#\k[\G]$ we have the following $\redbrv{u}{v}$ ordered by $\ell(uv)$, $u,v\in L$: If $\Sh{uv}{u}{v}$ then
\begin{align*}
\redbrv{1}{2}  &= \fvsl{12},   &    \redbrv{1}{112}  &= \redv{1112}  ,  & \redbrv{112}{12} &= \redv{11212}, \\
\redbrv{1}{12} &= \fvsl{112} , &    \redbrv{1}{122}  &= \redv{1122},    &   \redbrv{12}{122} &= \redv{12122},\\
\redbrv{12}{2} &= \fvsl{122} , &    \redbrv{122}{2} &= \redv{1222},
\end{align*}
and for $\nSh{1122}{112}{2}$ and $\nSh{112122}{112}{122}$ by Eq.~\eqref{RedCommutDefnVariableNotSh}
\begin{align*}
 \redbrv{112}{2} &= \partial_{1}^{\rho}(\redbrv{12}{2}) +q_{12,2}\redbrv{1}{2}\fvsl{12} -q_{1,12}\fvsl{12}\redbrv{1}{2}\\
                 &= \redv{1122} +(q_{12,2} -q_{1,12})\fvsl{12}^2,\\
\redbrv{112}{122}&= \partial_{1}^{\rho}(\redv{12122}) +q_{12,122}\redv{1122}\fvsl{12} -q_{1,12}\fvsl{12}\redv{1122}.
\end{align*}
We have to check the $q$-Jacobi conditions for $1<112<2$ (like Eq.~\eqref{ExB2JacobiCond1}), $1<112<12$ (like Eq.~\eqref{ExB2JacobiCond3}), $1<112<122$, $1<122<2$,
  $112<12<2$ (like Eq.~\eqref{ExB2JacobiCond2}), $112<12<122$, $112<122<2$,
  $12<122<2$  (note that we can omit $1<12<2$, $1<12<122$). The restricted $q$-Leibniz conditions are treated like before (note that we can leave out those for $1<112$, $1<12$, $1<122$ if $N_1<\infty$, $112 < 12$, $12 < 122$ if $N_{12}<\infty$, $112<2$, $12<2$, $122<2$ if $N_2<\infty$). 

Both types of conditions detect  many redundant relations like before. The proof that the given ideals of the  Nichols algebras and their liftings of \cite[Thm.~5.17 (1)]{Helbig-Lift} admit the PBW basis $\{x_1,[x_1x_1x_2],[x_1x_2],[x_1x_2x_2], x_2\}$ is again a straightforward but rather expansive calculation.

\subsection{PBW basis for $L=\{x_1  < x_1x_1x_2 < x_1x_1x_2x_1x_2 < x_1x_2   < x_2\}$}  \label{SectPBWbasisT5} 


This PBW basis  $[L]$ shows up in the Nichols algebras of non-standard type and their liftings of \cite[Thm.~5.17 (2),(4)]{Helbig-Lift}. More generally, we examine when $[L]$ is a PBW Basis of $(T(V)\#\k[\G])/I$ where $I$ is generated by
\begin{align*}
[x_1 x_1x_1x_2] &-\red{1112},                   &  x_1^{N_1}&-\redh{1},\\
[x_1x_1x_1x_2x_1x_2]&-\red{111212},             &  [x_1x_1x_2]^{N_{112}}&-\redh{112},\\
[x_1x_1x_2x_1x_1x_2x_1x_2]&-\red{11211212},     &  [x_1x_1x_2x_1x_2]^{N_{11212}}&-\redh{11212},\\
[x_1x_1x_2x_1x_2x_1x_2]&-\red{1121212},         &   [x_1x_2]^{N_{12}}&-\redh{12}, \\
[x_1x_2x_2]&-\red{122},                         &   x_2^{N_2}&-\redh{2}.
\end{align*}
 In $\k\la\fvsl 1,\fvsl{112},\fvsl{11212},\fvsl{12},\fvsl 2\ra\#\k[\G]$ we have the following $\redbrv{u}{v}$ ordered by $\ell(uv)$, $u,v\in L$: If $\Sh{uv}{u}{v}$ then
\begin{align*}
\redbrv{1}{2}  &= \fvsl{12},   &    \redbrv{1}{112}  &= \redv{1112}  ,  & \redbrv{11212}{12} &= \redv{1121212}, \\
\redbrv{1}{12} &= \fvsl{112} , &    \redbrv{112}{12}  &= \fvsl{11212},  & \redbrv{112}{11212} &= \redv{11211212},\\
\redbrv{12}{2} &= \redv{122} , &    \redbrv{1}{11212} &= \redv{111212},
\end{align*}
and for $\nSh{1122}{112}{2}$ and $\nSh{112122}{11212}{2}$ by Eq.~\eqref{RedCommutDefnVariableNotSh}
\begin{align*}
 \redbrv{112}{2} &= \partial_{1}^{\rho}(\redbrv{12}{2}) +q_{12,2}\redbrv{1}{2}\fvsl{12} -q_{1,12}\fvsl{12}\redbrv{1}{2}\\
                 &= \redv{1122} +(q_{12,2} -q_{1,12})\fvsl{12}^2,\\
\redbrv{11212}{2}&= \partial_{112}^{\rho}(\redv{122}) +q_{12,2}\redbrv{112}{2}\fvsl{12} -q_{112,12}\fvsl{12}\redbrv{112}{2}\\
             &= \partial_{112}^{\rho}(\redv{122}) 
+q_{12,2}\redv{1122}\fvsl{12}-q_{112,12}\fvsl{12}\redv{1122}\\
&\quad
+(q_{12,2}-q_{112,12})(q_{12,2} -q_{1,12})\fvsl{12}^3.
\end{align*}
Again we have to consider all $q$-Jacobi conditions and restricted $q$-Leibniz conditions, from where we detect again many redundant relations. Like before, we leave the concrete calculations for the cases of \cite[Thm.~5.17 (2),(4)]{Helbig-Lift} to the reader.

\subsection{PBW basis for $L=\{x_1 < x_1x_1x_1x_2 < x_1x_1x_2  < x_1x_2   < x_2\}$}  \label{SectPBWbasisT7} 


The Nichols algebras of non-standard type and their liftings in \cite[Thm.~5.17 (3),(5)]{Helbig-Lift} have this PBW basis $[L]$. We study the situation, when $[L]$ is a PBW Basis of $(T(V)\#\k[\G])/I$ where $I$ is generated by
\begin{align*}
[x_1 x_1x_1x_1x_2] &-\red{11112},                   &  x_1^{N_1}&-\redh{1},\\
[x_1x_1x_1x_2x_1x_1x_2]&-\red{1112112},             &  [x_1x_1x_1x_2]^{N_{1112}}&-\redh{1112},\\
[x_1x_1x_2x_1x_2]&-\red{11212},                  &  [x_1x_1x_2]^{N_{112}}&-\redh{112},\\
 [x_1x_2x_2]&-\red{122},           &   [x_1x_2]^{N_{12}}&-\redh{12}, \\
     &                   &   x_2^{N_2}&-\redh{2}.
\end{align*}
 In $\k\la\fvsl 1,\fvsl{112},\fvsl{11212},\fvsl{12},\fvsl 2\ra\#\k[\G]$ we have the following $\redbrv{u}{v}$ ordered by $\ell(uv)$, $u,v\in L$: If $\Sh{uv}{u}{v}$ then
\begin{align*}
\redbrv{1}{2}  &= \fvsl{12},   &    \redbrv{1}{112}  &= \fvsl{1112}  ,  & \redbrv{1112}{112} &= \redv{1121212}, \\
\redbrv{1}{12} &= \fvsl{112} , &    \redbrv{112}{12}  &= \redv{11212}, &&\\
\redbrv{12}{2} &= \redv{122} , &    \redbrv{1}{1112} &= \redv{11112},
\end{align*}
and for $\nSh{1122}{112}{2}$, $\nSh{11122}{1112}{2}$ and $\nSh{111212}{1112}{12}$ by Eq.~\eqref{RedCommutDefnVariableNotSh}
\begin{align*}
 \redbrv{112}{2} &= \partial_{1}^{\rho}(\redbrv{12}{2}) +q_{12,2}\redbrv{1}{2}\fvsl{12} -q_{1,12}\fvsl{12}\redbrv{1}{2}\\
                 &= \partial_{1}^{\rho}(\redv{122}) +(q_{12,2} -q_{1,12})\fvsl{12}^2,\\
\redbrv{1112}{2} &= \partial_{1}^{\rho}(\redbrv{112}{2}) +q_{112,2}\redbrv{1}{2}\fvsl{112} -q_{1,112}\fvsl{112}\redbrv{1}{2},\\
&=\partial_{1}^{\rho}(\partial_{1}^{\rho}(\redv{122}))  + (q_{12,2} -q_{1,12})(\fvsl{112}\fvsl{12}+q_{1,12}\fvsl{12}[\fvsl{1},\fvsl{12}])\\
&\quad +q_{112,2}\fvsl{12}\fvsl{112} -q_{1,112}\fvsl{112}\fvsl{12},\\
&=\partial_{1}^{\rho}(\partial_{1}^{\rho}(\redv{122}))+ q_{12}(q_{22} -q_{11}-q_{11}^2)\fvsl{112}\fvsl{12} \\
&\quad 
+q_{12}^2(q_{11}(q_{22} -q_{11})+q_{22})\fvsl{12}\fvsl{112},\\
\redbrv{1112}{12} &= \partial_{1}^{\rho}(\redv{11212}) +(q_{112,2} -q_{1,112})\fvsl{112}^2.
\end{align*}
Note that for the fifth equation we used the relation $[\fvsl{1},\fvsl{12}]-\fvsl{112}$. 
The assertion concerning the PBW basis and the redundant relations of \cite[Thm.~5.17 (3),(5)]{Helbig-Lift} are again straightforward to verify.

\bibliographystyle{plain}
\bibliography{mybib}
\end{document}